\documentclass[fleqn,preprint,3p,a4paper]{elsarticle}

\usepackage{amssymb}
\usepackage{amsmath}
\usepackage{amsthm}
\usepackage{dcolumn}
\usepackage{endnotes}
\usepackage{tabularx}
\usepackage[matrix,arrow]{xy}
\usepackage{wasysym}

\newtheorem{theorem}{Theorem}[section]
\newtheorem{proposition}[theorem]{Proposition}
\newtheorem{lemma}[theorem]{Lemma}
\newtheorem{corollary}[theorem]{Corollary}

\theoremstyle{definition}
\newtheorem{definition}[theorem]{Definition}
\newtheorem{example}[theorem]{Example}
\newtheorem{remark}[theorem]{Remark}
\newtheorem{question}[theorem]{Question}

\newcommand{\ir}{{\mathsf{Irr}}}

\newcommand{\mn}{\mathbb N}

\newcommand{\cl}{{\rm cl}}
\newcommand{\ii}{{\rm int}}

\newcommand{\ua}{\mathord{\uparrow}}
\newcommand{\da}{\mathord{\downarrow}}

\newcommand{\mk}{\mathord{\mathsf{K}}}
\newcommand{\wdd}{\mathord{\mathsf{WD}}}

\journal{}

\begin{document}

\begin{frontmatter}



\title{On Scott power spaces\tnoteref{t1}}
\tnotetext[t1]{This research was supported by the National Natural Science Foundation of China (Nos. 12071199, 12071188, 11661057)}.

\author[X. Xu]{Xiaoquan Xu\corref{mycorrespondingauthor}}
\cortext[mycorrespondingauthor]{Corresponding author}
\ead{xiqxu2002@163.com}
\address[X. Xu]{Fujian Key Laboratory of Granular Computing and Applications, Minnan Normal University, Zhangzhou 363000, China}
\author[X. Wen]{Xinpeng Wen}
\address[X. Wen]{College of Mathematics and Information, Nanchang Hangkong University, Nanchang 330063, China}
\ead{wenxinpeng2009@163.com}
\author[X. Xi]{Xiaoyong Xi}
\ead{littlebrook@jsnu.edu.cn}
\address[X. Xi]{School of Mathematics and Statistics, Yancheng Teachers University, Yancheng 224002, China}

\begin{abstract}
In this paper, we mainly discuss some basic properties of Scott power spaces. For a $T_0$ space $X$, let $\mathsf{K}(X)$ be the poset of all nonempty compact saturated subsets of $X$ endowed with the Smyth order. It is proved that the Scott power space $\Sigma \mathsf{K}(X)$ of a well-filtered space $X$ is still well-filtered, and a $T_0$ space $Y$ is well-filtered iff $\Sigma \mathsf{K}(Y)$ is well-filtered and the upper Vietoris topology is coarser than the Scott topology on $\mathsf{K}(Y)$. A sober space is constructed for which its Scott power space is not sober. A few sufficient conditions are given under which a Scott power space is sober. Some other properties, such as local compactness, first-countability, Rudin property and well-filtered determinedness, of Smyth power spaces and Scott power spaces are also investigated.
\end{abstract}

\begin{keyword}
Scott power space; Smyth power space; Sobriety; Well-filteredness; Local compactness; First-countability
\MSC 54B20; 54D99; 06B35; 06F30

\end{keyword}




\end{frontmatter}


\section{Introduction}

An important problem in domain theory is the modelling of non-deterministic features of programming languages and of parallel features treated in a non-deterministic way. If a non-deterministic program runs several times with the same input, it may produce different outputs. To describe this behaviour, powerdomains were introduced by Plotkin \cite{Plotkin1976, Plotkin1982} and Smyth \cite{Smyth1978} to give denotational semantics to non-deterministic choice in higher-order programming languages. The three main such powerdomains are the Smyth powerdomain for demonic non-determinism, the Hoare powerdomain for angelic non-determinism, and the Plotkin powerdomain for erratic non-determinism. This viewpoint traditionally stays with the category of dcpos, but is easily and profitably extended to general topological spaces (see, for example, \cite[Sections 6.2.3 and 6.2.4]{AJ94} and \cite{Schalk}).

A subset $A$ of a $T_0$ space $X$ is called \emph{saturated} if $A$ equals the intersection of all open sets containing it (or equivalently, $A$ is an upper set in the specialization order). We shall use $\mathsf{K}(X)$ to
denote the set of all nonempty compact saturated subsets of $X$ and endow it with the \emph{Smyth preorder}, that is, for $K_1,K_2\in \mathord{\mathsf{K}}(X)$, $K_1\sqsubseteq K_2$ iff $K_2\subseteq K_1$. The \emph{upper Vietoris topology} on $\mathsf{K}(X)$ is the topology that has $\{\Box U : U\in \mathcal O(X)\}$ as a base, where $\Box U=\{K\in \mathsf{K}(X) : K\subseteq  U\}$, and the resulting space is called the \emph{Smyth power space} or \emph{upper space} of $X$ and is denoted by $P_S(X)$.

There is another prominent topology one can put on $\mathsf{K}(X)$, namely, the Scott topology. We call the space $\Sigma \mathsf{K}(X)=(\mathsf{K}(X), \sigma (\mathsf{K}(X)))$ the Scott power space of $X$. It is well-known that when $X$ is well-filtered, $\mathsf{K}(X)$ is a dcpo, with least upper bounds
of directed families computed as filtered intersections, and the Scott topology is finer than the upper Vietoris topology; when $X$
is locally compact and well-filtered (equivalently, locally compact and well-filtered), the two topologies coincide.

In this paper, we mainly discuss some basic properties of Scott power spaces. The paper is organized as follows:

In Section 2, some standard definitions and notations are introduced which will be used in the whole paper. A few basic properties of irreducible sets and compact saturated sets are listed.

In Section 3, we briefly recall the concepts of Scott topology and continuous domains and some fundamental results about them.  A countable algebraic lattice $L$ is given for which the poset $\mathbf{Fin}L$ of all upper finitely generated sets of $L$ (with the reverse inclusion order) is not a dcpo.

In Section 4, we list a few important results of $d$-spaces, well-filtered spaces and sober spaces that will be used in other sections.

In Section 5, we recall some concepts and results about the topological Rudin Lemma, Rudin spaces and
well-filtered determined spaces that will be used in the next four sections.

In Section 6, we mainly investigate the well-filteredness of Scott power spaces. It is proved that the Scott power space of a well-filtered space is well-filtered, and a $T_0$ space $X$ is well-filtered iff the upper Vietoris topology is coarser than the Scott topology on $\mathsf{K}(X)$ and $\Sigma \mathsf{K}(X)$ is well-filtered.

In Section 7, a sober space is constructed for which its Scott power space is not sober.

In Section 8, we study the question under what conditions the Scott power space $\Sigma \mathsf{K}(X)$ of a sober space $X$ is sober. This question is related to the investigation of conditions under which the upper Vietoris topology coincides with the Scott topology on $\mathsf{K}(X)$, and further it is closely related to the local compactness and first-countability of $X$.

In section 9, the Rudin property and well-filtered determinedness of Smyth power spaces and Scott power spaces are discussed.

\section{Preliminaries}

In this section, we briefly recall some standard definitions and notations that will be used in the paper. Some basic properties of irreducible sets and compact saturated sets are presented.

For a set $X$, $|X|$ will denote the cardinality of $X$. Let $\mathbb{N}$ denote the set of all natural numbers with the usual order if no other explanation and $\omega=|\mn|$. The set of all subsets of $X$ is denoted by $2^X$. Let $X^{(<\omega)}=\{F\subseteq X : F \mbox{~is a finite set}\}$ and $X^{(\leqslant\omega)}=\{F\subseteq X : F \mbox{~is a countable set}\}$.

For a poset $P$ and $A\subseteq P$, let
$\mathord{\downarrow}A=\{x\in P: x\leq  a \mbox{ for some }
a\in A\}$ and $\mathord{\uparrow}A=\{x\in P: x\geq  a \mbox{
	for some } a\in A\}$. For  $x\in P$, we write
$\mathord{\downarrow}x$ for $\mathord{\downarrow}\{x\}$ and
$\mathord{\uparrow}x$ for $\mathord{\uparrow}\{x\}$. A subset $A$
is called a \emph{lower set} (resp., an \emph{upper set}) if
$A=\mathord{\downarrow}A$ (resp., $A=\mathord{\uparrow}A$). Let $\mathbf{Fin} P=\{\uparrow F : F\in P^{(<\omega)}\}$.  For a nonempty subset $A$ of $P$, define $\mathrm{min} (A)=\{u\in A : u$ is a minimal element of  $A\}$ and $\mathrm{max} (A)=\{v\in A : v$ is a maximal element of  $A\}$.

A nonempty subset $D$ of a poset $P$ is \emph{directed} if every two
elements in $D$ have an upper bound in $D$. The set of all directed sets of $P$ is denoted by $\mathcal D(P)$. $I\subseteq P$ is called an \emph{ideal} of $P$ if $I$ is a directed lower subset of $P$. Let $\mathrm{Id} (P)$ be the poset (with the order of set inclusion) of all ideals of $P$. Dually, we define the notion of \emph{filters} and denote the poset of all filters of $P$ by $\mathrm{Filt}(P)$. The poset $P$ is called a
\emph{directed complete poset}, or \emph{dcpo} for short, if for any
$D\in \mathcal D(P)$, $\vee D$ exists in $P$.

The poset $P$ is said to be \emph{Noetherian} if it satisfies the \emph{ascending chain condition} ($\mathrm{ACC}$ for short): every ascending chain has a greatest member. Clearly, $P$ is Noetherian if{}f every directed set of $P$ has a largest element (or equivalently, every ideal of $P$ is principal). A topological space $X$ is said to be a \emph{Noetherian space} if every open subset is compact (see \cite[Definition 9.7.1]{Jean-2013}).

As in \cite{E_20182}, a topological space $X$ is \emph{locally hypercompact} if for each $x\in X$ and each open neighborhood $U$ of $x$, there is  $\ua F\in \mathbf{Fin}X$ such that $x\in\ii\,\ua F\subseteq\ua F\subseteq U$. The space $X$ is called a $c$-\emph{space} if for each $x\in X$ and each open neighborhood $U$ of $x$, there is $u\in X$ such that $x\in\ii\,\ua u\subseteq\ua u\subseteq U$). A set $K$ of $X$ is called \emph{supercompact} if for
any family $\{U_i : i\in I\}\subseteq \mathcal O(X)$, $K\subseteq \bigcup_{i\in I} U_i$  implies $K\subseteq U$ for some $i\in I$. It is easy to verify that the supercompact saturated sets of $X$ are exactly the sets $\ua x$ with $x \in X$ (see \cite[Fact 2.2]{Klause-Heckmann}). It is well-known that $X$ is a $c$-space if{}f $\mathcal O(X)$ is a \emph{completely distributive} lattice (cf. \cite{E_2009}).

The category of all $T_0$ spaces is denoted by $\mathbf{Top}_0$. For $X\in ob(\mathbf{Top}_0)$, we use $\leq_X$ to denote the \emph{specialization order} of $X$: $x\leq_X y$ if{}f $x\in \overline{\{y\}}$). In the following, when a $T_0$ space $X$ is considered as a poset, the order always refers to the specialization order if no other explanation. Let $\mathcal O(X)$ (resp., $\mathcal C(X)$) be the set of all open subsets (resp., closed subsets) of $X$, and let $\mathcal S^u(X)=\{\ua x : x\in X\}$. Define $\mathcal S_c(X)=\{\overline{{\{x\}}} : x\in X\}$ and $\mathcal D_c(X)=\{\overline{D} : D\in \mathcal D(X)\}$.

It is straightforward to verify the following.

\begin{remark}\label{closure A = closre B}   Let $X$ be a topological space and $A, B\subseteq X$. Then
\begin{enumerate}[\rm (1)]
\item $\overline{A}=\overline{B}$ if and only if for any $U\in \mathcal O(X)$, $A\cap U\neq \emptyset$ iff $B\cap U\neq\emptyset$.
\item If $\tau_1, \tau_2$ are two topologies on the set $X$ and $\tau_1\subseteq \tau_2$, then $\cl_{\tau_2}A=\cl_{\tau_2}B$ implies  $\cl_{\tau_1}A=\cl_{\tau_1}B$.
\end{enumerate}
\end{remark}

For a $T_0$ space $X$ and a nonempty subset $A$ of $X$, $A$ is \emph{irreducible} if for any $\{F_1, F_2\}\subseteq \mathcal C(X)$, $A \subseteq F_1\cup F_2$ implies $A \subseteq F_1$ or $A \subseteq  F_2$.  Denote by $\ir(X)$ (resp., $\ir_c(X)$) the set of all irreducible (resp., irreducible closed) subsets of $X$. Clearly, every subset of $X$ that is directed under $\leq_X$ is irreducible.

The following lemma is well-known and can be easily verified.

\begin{lemma}\label{irrimage}
	If $f : X \longrightarrow Y$ is continuous and $A\in\ir (X)$, then $f(A)\in \ir (Y)$.
\end{lemma}

For any $T_0$ space $X$, the \emph{lower Vietoris topology} on $\ir_c(X)$ is the topology $\{\Diamond U : U\in \mathcal O(X)\}$, where $\Diamond  U=\{A\in \ir_c(X): A\cap U\neq\emptyset\}$. The resulting space, denoted by $X^s$, with the canonical mapping $\eta_{X}: X\longrightarrow X^s, x\mapsto\overline {\{x\}}$, is the \emph{sobrification} of $X$ (cf. \cite{redbook, Jean-2013}). Clearly, $\eta_{X}: X\longrightarrow X^s$ is an order and topological embedding (cf. \cite{redbook, Jean-2013, Schalk}).

\begin{remark} \label{eta continuous} For a $T_0$ space $X$, $\eta_{X}: X\longrightarrow X^s$ is a dense topological embedding (cf. \cite{redbook, Jean-2013, Schalk}).
\end{remark}

 A subset $A$ of a space $X$ is called \emph{saturated} if $A$ equals the intersection of all open sets containing it (or equivalently, $A$ is an upper set in the specialization order). We shall use $\mathsf{K}(X)$ to
denote the set of all nonempty compact saturated subsets of $X$ and endow it with the \emph{Smyth preorder}, that is, for $K_1,K_2\in \mathord{\mathsf{K}}(X)$, $K_1\sqsubseteq K_2$ if{}f $K_2\subseteq K_1$. The \emph{upper Vietoris topology} on $\mathsf{K}(X)$ is the topology that has $\{\Box U : U\in \mathcal O(X)\}$ as a base, where $\Box U=\{K\in \mathsf{K}(X) : K\subseteq  U\}$, and the resulting space is called the \emph{Smyth power space} or \emph{upper space} of $X$ and is denoted by $P_S(X)$ (cf. \cite{Heckmann-1992, Schalk}).

\begin{remark}\label{xi embdding} Let $X$ be a $T_0$ space.
\begin{enumerate}[\rm (1)]
\item The specialization order on $P_S(X)$ is the Smyth order (that is, $\leq_{P_S(X)}=\sqsubseteq$).
\item The \emph{canonical mapping} $\xi_X: X\longrightarrow P_S(X)$, $x\mapsto\ua x$, is an order and topological embedding
(cf. \cite{Heckmann-1992, Klause-Heckmann, Schalk}).
\item $X$ is homeomorphic to the subspace $S^u(X)$ of $P_S(X)$ by means of $\xi_X$.

\end{enumerate}
\end{remark}

\begin{lemma}\label{closure in Smyth power space} For a $T_0$ space $X$ and $A\subseteq X$, $\cl_{\mathcal O(P_S(X))}\xi_X(A)=\Diamond \overline{A}$.
\end{lemma}
\begin{proof} Clearly, $\Diamond \overline{A}=\mathsf{K}(X)\setminus \Box (X\setminus \overline{A})$ is closed in $P_S(X)$ and hence $\cl_{\mathcal O(P_S(X))}\xi_X(A)\subseteq\Diamond \overline{A}$. Since $\{\Diamond C : C\in \mathcal C(X)\}$ is a (closed) base of $P_S(X)$, there is a family $\{C_i : i\in I\}\subseteq \mathcal C(X)$ such that $\cl_{\mathcal O(P_S(X))}\xi_X(A)=\bigcap_{i\in I} \Diamond C_i$. Then for each $i\in I$, $\xi_X(A)\subseteq \Diamond C_i$, and consequently, $\ua a\cap C_i\neq \emptyset$ for each $a\in A$; whence, for each $a\in A$, $a\in C_i$ as $C_i=\da C_i$. It follow that $\overline{A}\subseteq C_i$ for each $i\in I$ and hence $\Diamond \overline{A}\subseteq \bigcap_{i\in I}\Diamond C_i=\cl_{\mathcal O(P_S(X))}\xi_X(A)$. Thus $\cl_{\mathcal O(P_S(X))}\xi_X(A)=\Diamond \overline{A}$.

\end{proof}

\begin{proposition}\label{PS functor} (\cite[Lemma 2.19]{Xu2021}) $P_S : \mathbf{Top}_0 \longrightarrow \mathbf{Top}_0$ is a covariant functor, where for any $f : X \longrightarrow Y$ in $\mathbf{Top}_0$, $P_S(f) : P_S(X) \longrightarrow P_S(Y)$ is defined by $P_S(f)(K)=\uparrow f(K)$ for all $ K\in\mathsf{K}(X)$.
\end{proposition}

\begin{corollary}\label{PS retract}  Let $X$ and $Y$ be two $T_0$ spaces. If $Y$ is a retract of $X$, then $P_S(Y)$ is a retract of $P_S(X)$.
\end{corollary}

For a nonempty subset $C$ of a $T_0$ $X$, it is easy to see that $C$ is compact if{}f $\ua C\in \mk (X)$. Furthermore, we have the following useful result (see, e.g., \cite[pp.2068]{E_2009}).

\begin{lemma}\label{COMPminimalset} Let $X$ be a $T_0$ space and $C\in \mk (X)$. Then $C=\ua \mathrm{min}(C)$ and  $\mathrm{min}(C)$ is compact.
\end{lemma}

\begin{lemma}\label{Kmeet} Let $X$ be a $T_0$ space. For any nonempty family $\{K_i : i\in I\}\subseteq \mk (X)$, $\bigvee_{i\in I} K_i$ exists in $\mk (X)$ if{}f~$\bigcap_{i\in I} K_i\in \mk (X)$. In this case $\bigvee_{i\in I} K_i=\bigcap_{i\in I} K_i$.
\end{lemma}
\begin{proof} Suppose that $\{K_i : i\in I\}\subseteq \mk (X)$ is a nonempty family and $\bigvee_{i\in I} K_i$ exists in $\mk (X)$. Let $K=\bigvee_{i\in I} K_i$. Then $K\subseteq K_i$ for all $i\in I$, and hence $K\subseteq \bigcap_{i\in I} K_i$. For any $x\in \bigcap_{i\in I} K_i$, $\ua x$ is a upper bound of $\{K_i : i\in I\}\subseteq \mk (X)$, whence $K\sqsubseteq \ua x$ or, equivalently, $\ua x \subseteq K$. Therefore, $\bigcap_{i\in I} K_i\subseteq K$. Thus $\bigcap_{i\in I} K_i=K\in \mk (X)$.

Conversely, if $\bigcap_{i\in I} K_i\in \mk (X)$, then $\bigcap_{i\in I} K_i$ is an upper bound of $\{K_i : i\in I\}$ in $\mk (X)$. Let $G\in \mk (X)$ be another upper bound of $\{K_i : i\in I\}$, then $G\subseteq K_i$ for all $i\in I$, and hence $G\subseteq \bigcap_{i\in I} K_i$, that is, $\bigcap_{i\in I} K_i\sqsubseteq G$, proving that $\bigvee_{i\in I} K_i=\bigcap_{i\in I} K_i$.
\end{proof}

Similarly, we have the following.

\begin{lemma}\label{Fin P sup} Let $P$ be a poset. For any nonempty family $\{\ua F_i : i\in I\}\subseteq \mathbf{Fin}~\!P$, $\bigvee_{i\in I} \ua F_i$ exists in $\mathbf{Fin}~\!P$  if{}f~$\bigcap_{i\in I} \ua F_i\in \mathbf{Fin}~\!P$. In this case $\bigvee_{i\in I} \ua F_i=\bigcap_{i\in I} \ua F_i$.
\end{lemma}

\begin{lemma}\label{K union} (\cite[Proposition 7.21]{Schalk})  Let $X$ be a $T_0$ space.
\begin{enumerate}[\rm (1)]
\item If $\mathcal K\in\mk(P_S(X))$, then $\bigcup \mathcal K\in\mk(X)$.
\item The mapping $\bigcup : P_S(P_S(X)) \longrightarrow P_S(X)$, $\mathcal K\mapsto \bigcup \mathcal K$, is continuous.
\end{enumerate}
\end{lemma}

For a metric space $(X, d)$, $x\in X$ and a positive number $r$, let $B(x, r)=\{y\in Y : d(x, y)<r\}$ be the $r$-\emph{ball} about $x$. For a set $A\subseteq X$ and a positive number $r$, by the $r$-\emph{ball} about $A$ we mean the set $B(A, r)=\bigcup_{a\in A}B(a, r)$.

The following result is well-known (cf. \cite{Engelking}).

\begin{lemma}\label{metric space compact sets}  Let $(X, d)$ be a metric space and $K$ a compact set of $X$. Then for any open set $U$ containing $K$, there is an $r>0$ such that $K\subseteq B(K, r)\subseteq U$.
\end{lemma}

\section{Scott topology and continuous domains}

For a poset $P$, a subset $U$ of $P$ is \emph{Scott open} if
(i) $U=\mathord{\uparrow}U$ and (ii) for any directed subset $D$ with
$\vee D$ existing, $\vee D\in U$ implies $D\cap
U\neq\emptyset$. All Scott open subsets of $P$ form a topology,
called the \emph{Scott topology} on $P$ and
denoted by $\sigma(P)$. The space $\Sigma~\!\! P=(P,\sigma(P))$ is called the
\emph{Scott space} of $P$. For the chain $2=\{0, 1\}$ (with the order $1<2$), we have $\sigma(2)=\{\emptyset, \{1\}, \{0,1\}\}$. The space $\Sigma~\!\!2$
is well-known under the name of \emph{Sierpi\'nski space}. The \emph{upper topology} on $P$, generated
by the complements of the principal ideals of $P$, is denoted by $\upsilon (P)$.
The upper sets of $P$ form the (\emph{upper}) \emph{Alexandroff topology} $\alpha (P)$.

\begin{lemma}\label{Scott continuous equiv} (\cite[Proposition II-2.1]{redbook}) For posets $P, Q$ and $f : P \longrightarrow Q$, the following two conditions are equivalent:
\begin{enumerate}[\rm (1)]
\item $f : \Sigma~\!\! P \longrightarrow \Sigma~\!\! Q$ is continuous.
\item For each $D\in \mathcal D(P)$ for which $\vee D$ exists in $P$, $\vee f(D)$ exists in $Q$ and $f(\vee D)=\vee f(D)$.
\end{enumerate}
\end{lemma}

For a dcpo $P$ and $A, B\subseteq P$, we say $A$ is \emph{way below} $B$, written $A\ll B$, if for each $D\in \mathcal D(P)$, $\vee D\in \ua B$ implies $D\cap \ua A\neq \emptyset$. For $B=\{x\}$, a singleton, $A\ll B$ is
written $A\ll x$ for short. For $x\in P$, let $w(x)=\{F\in P^{(<\omega)} : F\ll x\}$, $\Downarrow x = \{u\in P : u\ll x\}$ and $K(P)=\{k\in P : k\ll k\}$. Points in $K(P)$ are called \emph{compact elements} of $P$.

For the following definition and related conceptions, please refer to \cite{redbook}.

\begin{definition}\label{continuous domain etc} Let $P$ be a dcpo and $X$ a $T_0$ space.
\begin{enumerate}[\rm (1)]
\item $P$ is called a \emph{continuous domain}, if for each $x\in P$, $\Downarrow x$ is directed
and $x=\vee\Downarrow x$. When a complete
lattice $L$ is continuous, we call $L$ a \emph{continuous lattice}.
\item  $P$ is called an \emph{algebraic domain}, if for each $x\in P$, $\{k\in K(P) : k\leq x\}$ is directed
and $x=\vee \{k\in K(P) : k\leq x\}$. When a complete
lattice $L$ is algebraic, we call $L$ an \emph{algebraic lattice}.
\item $P$ is called a \emph{quasicontinuous domain}, if for each $x\in P$, $\{\ua F : F\in w(x)\}$ is filtered and $\ua x=\bigcap
\{\ua F : F\in w(x)\}$.
\item $X$ is called \emph{core-compact} if $\mathcal O(X)$ is a \emph{continuous lattice}.
\end{enumerate}
\end{definition}

\begin{remark}\label{core-compact is not locally compact}  It is well-known that if a topological space $X$ is locally compact, then it is core-compact (see, e.g., \cite[Examples I-1.7]{redbook}). In \cite[Section 7]{Hofmann-Lawson-1978} (see also \cite[Exercise V-5.25]{redbook}) Hofmann and Lawson gave a second-countable core-compact $T_0$ space $X$ in which every compact subset of $X$ has empty interior and hence it is not locally compact.
\end{remark}

The following result is well-known (see \cite{redbook}).

\begin{theorem}\label{algebraic is continuous} Let $P$ be a dcpo.
\begin{enumerate}[\rm (1)]
\item If $P$ is algebraic, then it is continuous.
\item If $P$ is continuous, then it is quasicontinuous.
\item $P$ is continuous if{}f $\Sigma ~\!\! P$ is a $c$-space.
\item $P$ is quasicontinuous if{}f $\Sigma ~\!\! P$ is locally hypercompact.
\end{enumerate}
\end{theorem}

The following example show that there is even a countable algebraic lattice $L$ such that $\mathbf{Fin}~\!L$ (note that the order on $\mathbf{Fin}~\!L$ is the reverse inclusion
order $\supseteq$) is not a dcpo.

\begin{example}\label{algebraic lattice with non-dcpo Fin L}
Let $L=\mathbb{N}\cup \{\omega_{0},\omega_{1},\cdot \cdot \cdot,\omega_{n},\cdot \cdot \cdot \}\cup \{\infty\}$. Define an order on $L$ as follows (see Figure 1):
\begin{enumerate}[\rm (i)]
\item $1<2<3<\cdot\cdot\cdot <n<n+1<\cdot\cdot\cdot <\omega_{0}$,

\item $n<\omega_{n}$ for each $n\in \mathbb{N}$, and

\item $\infty$ is a largest element of $L$.
\end{enumerate}
\noindent It is easy to see that $L$ is a complete lattice. Moreover, for each $x\in L\setminus \{\omega_{0}\}$, $x\ll x$, whence $L$ is a countable  algebraic lattice.

Now we show that $\mathbf{Fin}~\!L$ is not a dcpo. Let
\begin{eqnarray}
\ua F_{n}&=&\ua n\cup \ua\{\omega_{1},\cdot\cdot\cdot ,\omega_{n-1}\}\nonumber\\
&=& \ua n\cup\{\omega_{1},\cdot\cdot\cdot ,\omega_{n-1}\}\cup\{\infty\}\nonumber\\
&=&\{n,n+1,\cdot\cdot\cdot\}\cup \{\omega_{0},\omega_{1},\cdot\cdot\cdot ,\omega_{n},\cdot\cdot\cdot\}\cup \{\infty\}.\nonumber
\end{eqnarray}

\begin{figure}[ht]
	\centering
	\includegraphics[height=5cm,width=6cm]{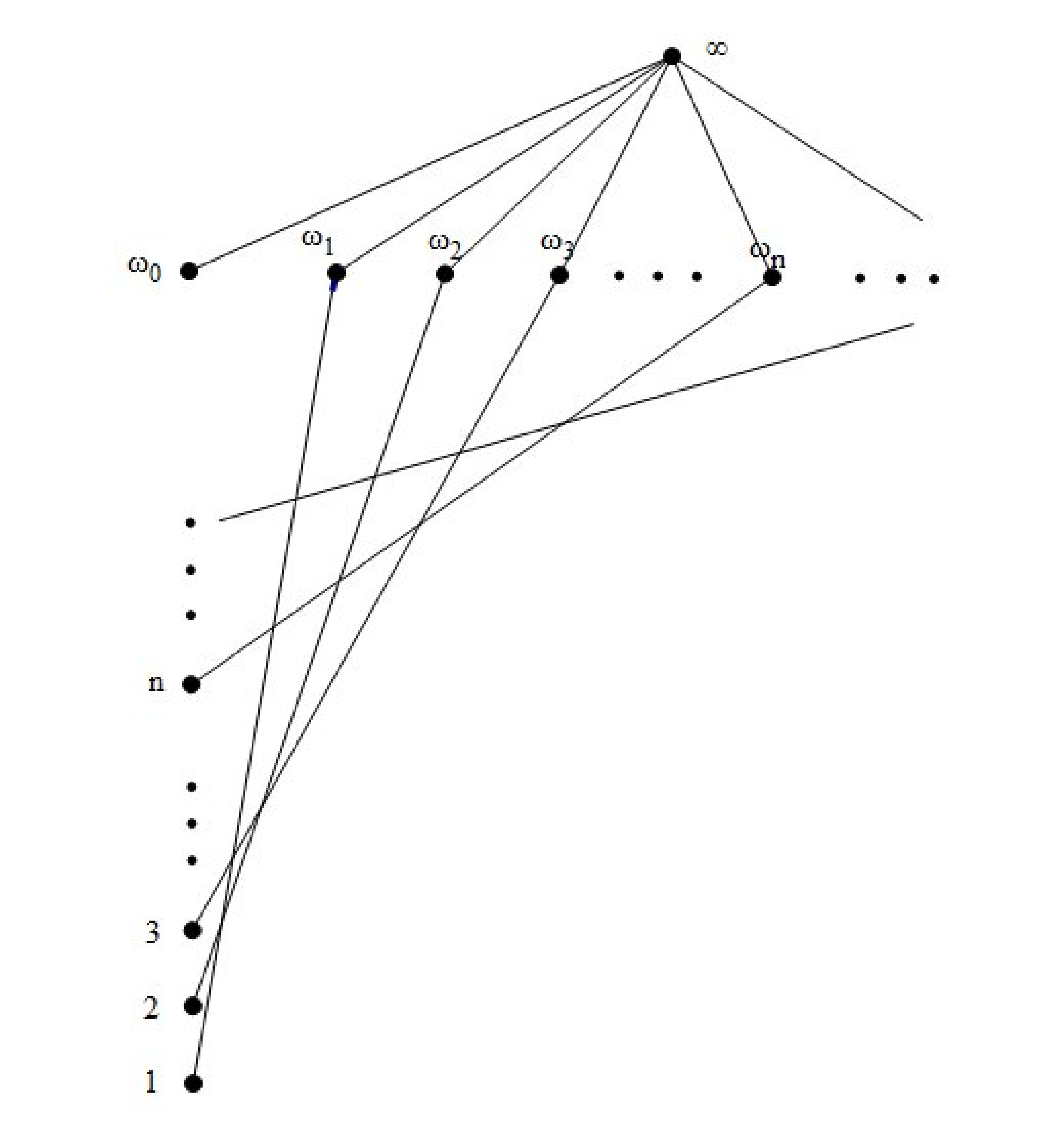}
	\caption{An algebraic lattice $L$ with non-dcpo $\mathbf{Fin}~\!L$}
\end{figure}

It is clear that $\{\ua F_{n}:n\in \mathbb{N}\}\subseteq \mathcal D(\mathbf{Fin}~\!L)$ and $$\bigcap_{n\in \mathbb{N}}\ua F_{n}=\{\omega_{0},\omega_{1},\cdot\cdot\cdot,\omega_{n},\cdot\cdot\cdot\}\cup \{\infty\}\not\in \mathbf{Fin}~\!L.$$

\noindent By Lemma \ref{Fin P sup}, $\{\ua F_{n}:n\in \mathbb{N}\}$ has no join in $\mathbf{Fin}~\!L$. Thus $\mathbf{Fin}~\!L$ is not a dcpo.
\end{example}
As $\mathbf{Fin}~\!L$ is not a dcpo, it can not be used as a mathematical model for denotational semantics of non-deterministic programs. So one should look for other novel mathematical structures, such as the poset of all nonempty compact saturated sets of a suitable $T_0$ space $X$ and the poset of all strongly compact saturated sets of $X$ ($ S\subseteq X$ is strongly compact if for all open sets $U$ with $S \subseteq U$,
there is a finite set $F$  with $S\subseteq \uparrow  F \subseteq U$) (cf. \cite{Heckmann-1992}).
\section{$d$-spaces, well-filtered spaces and sober spaces}

A $T_0$ space $X$ is called a $d$-\emph{space} (or \emph{monotone convergence space}) if $X$ (with the specialization order) is a dcpo
 and $\mathcal O(X) \subseteq \sigma(X)$ (cf. \cite{redbook, Wyler}).

It is easy to verify the following result (cf. \cite{redbook, XSXZ-2020}).

\begin{proposition}\label{d-spacecharac1} For a $T_0$ space $X$, the following conditions are equivalent:
\begin{enumerate}[\rm (1) ]
	        \item $X$ is a $d$-space.
            \item $\mathcal D_c(X)=\mathcal S_c(X)$.
            \item  $X$ is a dcpo, and $\overline{D}=\overline{\{\vee D\}}$ for any $D\in \mathcal D(X)$.
            \item  For any $D\in \mathcal D(X)$ and $U\in \mathcal O(X)$, $\bigcap\limits\limits_{d\in D}\ua d\subseteq U$ implies $\ua d \subseteq U$ (i.e., $d\in U$) for some $d\in D$.

\end{enumerate}
\end{proposition}

\begin{lemma}\label{d-space max point} (\cite[Lemma 2.1]{xuzhao-20202} )
Let $X$ be a $d$-space. Then for any nonempty closed subset $A$ of $X$,  $A=\da \mathrm{max}(A)$, and hence $\mathrm{max}(A)\neq\emptyset$.
\end{lemma}

A topological space $X$ is called \emph{sober}, if for any  $F\in\ir_c(X)$, there is a unique point $a\in X$ such that $F=\overline{\{a\}}$. Hausdorff spaces are always sober (see, e.g., \cite[Proposition 8.2.12]{Jean-2013}) and sober spaces are always $T_0$ since $\overline{\{x\}}=\overline{\{y\}}$ always implies $x=y$. The Sierpinski space $\Sigma~\!\!2$ is sober but not $T_1$ and
an infinite set with the co-finite topology is $T_1$ but not sober (see Example \ref{X d-space not imply Scott power space d-space}).

The following conclusion is well-known (see, e.g., \cite{redbook, quasicont, Heckmann-1992}).

\begin{proposition}\label{quasicontinuous domain is sober} For a quasicontinuous domain $P$, $\Sigma ~\!\! P$ is sober.
\end{proposition}

For the sobriety of the Smyth power spaces, we have the following well-known result.

\begin{theorem}\label{Schalk-Heckman-Keimel theorem}(Heckmann-Keimel-Schalk Theorem) (\cite[Theorem 3.13]{Klause-Heckmann}, \cite[Lemma 7.20]{Schalk}) For a $T_0$ space $X$, the following conditions are equivalent:
\begin{enumerate}[\rm (1)]
\item $X$ is sober.
 \item  For any $\mathcal A\in \ir(P_S(X))$ and $U\in \mathcal O(X)$, $\bigcap\mathcal K\subseteq U$ implies $K \subseteq U$ for some $K\in \mathcal A$.
 \item $P_S(X)$ is sober.
\end{enumerate}
\end{theorem}

A $T_0$ space $X$ is called \emph{well-filtered} if for any filtered family $\mathcal{K}\subseteq \mathord{\mathsf{K}}(X)$ and open set $U$, $\bigcap\mathcal{K}{\subseteq} U$ implies $K{\subseteq} U$ for some $K{\in}\mathcal{K}$.

\begin{remark}\label{sober implies WF implies d-space} The following implications are well-known (which are irreversible) (cf. \cite{redbook}):

$$\mbox{sobriety $\Rightarrow$ well-filteredness $\Rightarrow$ $d$-space.}$$
\end{remark}

In \cite{Xi-Lawson-2017} and \cite{Hofmann-Lawson-1978, Kou}, the following two useful results were given.

\begin{proposition}\label{Xi-Lawson result 1} (\cite[Corollary 3.2]{Xi-Lawson-2017})
If a dcpo $P$ endowed with the Lawson topology is compact (in particular, $P$ is a complete lattice), then $\Sigma P$ is well-filtered.
\end{proposition}

\begin{theorem}\label{SoberLC=CoreC}(\cite[Corollary 4.6]{Hofmann-Lawson-1978}, \cite[Theorem 2.3]{Kou})  For a $T_0$ space $X$, the following conditions are equivalent:
\begin{enumerate}[\rm (1)]
	\item $X$ locally compact and sober.
	\item $X$ is locally compact and well-filtered.
	\item $X$ is core-compact and sober.
\end{enumerate}
\end{theorem}

For the well-filteredness of topological spaces, a similar result to Theorem \ref{Schalk-Heckman-Keimel theorem} was proved in \cite{xuxizhao} (see also \cite{XSXZ-2020}).

\begin{theorem}\label{Smythwf} (\cite[Theorem 5.3]{XSXZ-2020}, \cite[Theorem 4]{xuxizhao})
	For a $T_0$ space, the following conditions are equivalent:
\begin{enumerate}[\rm (1)]
		\item $X$ is well-filtered.
        \item $P_S(X)$ is a $d$-space.
        \item $P_S(X)$ is well-filtered.
\end{enumerate}
\end{theorem}

\begin{corollary}\label{wf space Vietoris less Scott}
	For a well-filtered space (especially, a sober space) $X$, $\mathsf{K}(X)$ (with the Smyth order) is a dcpo and the upper Vietoris topology is coarser than the Scott topology on $\mathsf{K}(X)$.
\end{corollary}

By Theorem \ref{Smythwf} and Corollary \ref{wf space Vietoris less Scott} we know that for a $T_0$ space $X$, if $P_S(X)$ is a $d$-space (equivalently, $X$ is a well-filtered space), then $\Sigma~\!\!\mathsf{K}(X)$ is a $d$-space. Example \ref{Scott sober not implies X is wf} below shows that $\Sigma~\!\!\mathsf{K}(X)$ is a sober space does not imply that $X$ is well-filtered (i.e., $P_S(X)$ is a $d$-space) in general.

The following example shows that there is a $T_0$ space $X$ such that $\mathsf{K}(X)$ (with the Smyth order) is a dcpo but $X$ is not well-filtered.

\begin{example}\label{K(X) is dcpo not implies X is wf} (Johnstone's dcpo adding a top element)  Let $\mathbb{J}=\mathbb{N}\times (\mathbb{N}\cup \{\infty\})$ with ordering defined by $(j, k)\leq (m, n)$ if{}f $j = m$ and $k \leq n$, or $n =\infty$ and $k\leq m$. $\mathbb{J}$ is a well-known dcpo constructed by Johnstone in \cite{johnstone-81} (see Figure 2).

\begin{figure}[ht]
	\centering
	\includegraphics[height=4.5cm,width=4.5cm]{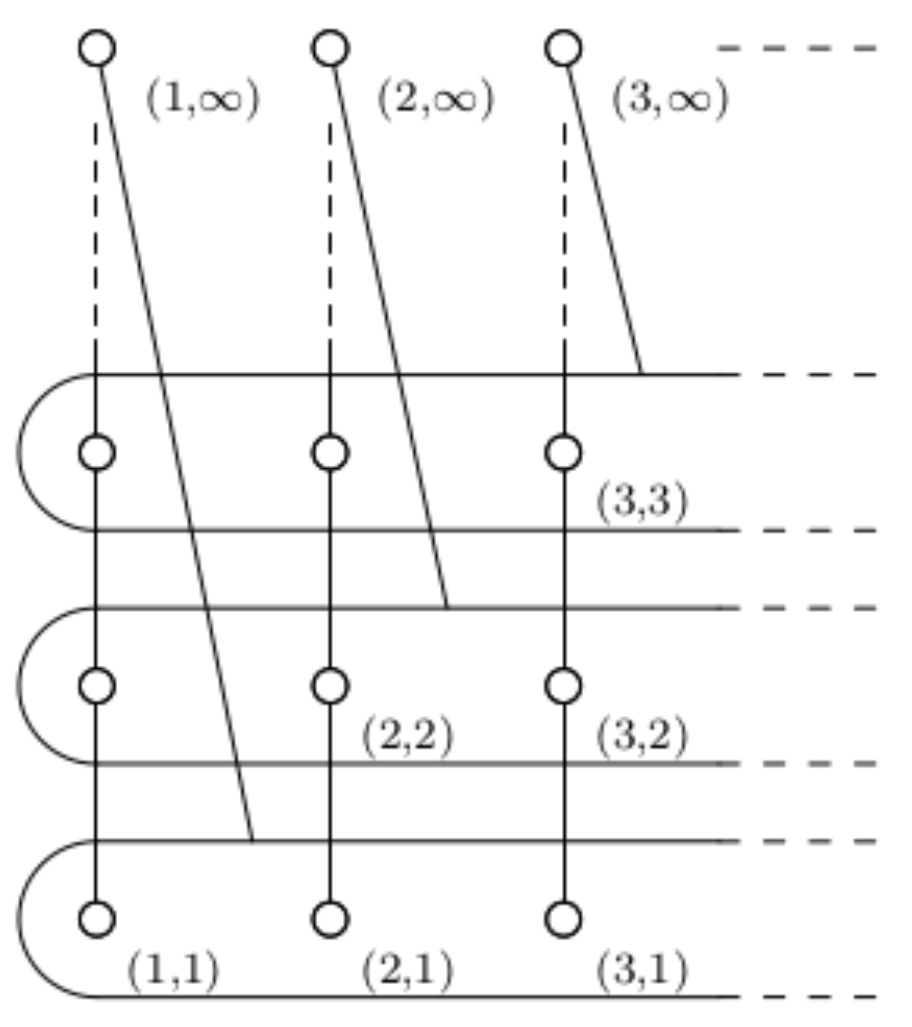}
	\caption{Johnstone's dcpo $\mathbb{J}$}
\end{figure}

\noindent The set $\mathbb{J}_{max}=\{(n, \infty) : n\in\mn \}$ is the set of all maximal elements of $\mathbb{J}$. Adding top $\top$ to $\mathbb{J}$  yields a dcpo $\mathbb{J}_{\top}=\mathbb{J}\cup \{\top\}$ ($x\leq \top$ for any $x\in \mathbb{J}$). Then $\top$ is the largest element of $\mathbb{J}_{\top}$ and $\{\top\}\in \sigma (\mathbb{J}_{\top})$. The following three conclusions about $\Sigma~\!\!\mathbb{J}$ are known (see, for example, \cite[Example 3.1]{LL-2017} and \cite[Lemma 3.1]{MLZ-2021}):
\begin{enumerate}[\rm (i)]
\item $\ir_c (\Sigma~\!\!\mathbb{J})=\{\overline{\{x\}}=\da_{\mathbb{J}} x : x\in \mathbb{J}\}\cup \{\mathbb{J}\}$.
\item $\mathsf{K}(\Sigma~\!\!\mathbb{J})=(2^{\mathbb{J}_{max}} \setminus \{\emptyset\})\bigcup \mathbf{Fin}~\!\mathbb{J}$.
\item $\Sigma~\!\!\mathbb{J}$ is not well-filtered.
\end{enumerate}
\noindent Hence we have
\begin{enumerate}[\rm (a)]
\item $\ir_c (\Sigma~\!\!\mathbb{J}_\top)=\{\overline{\{x\}}=\da_{\mathbb{J}_{\top}} x : x\in \mathbb{J}_\top\}\cup \{\mathbb{J}\}$ by (i).
\item $\mathsf{K}(\Sigma~\!\!\mathbb{J}_\top)=\{\ua_{\mathbb{J}_{\top}} G : G \mbox{~is nonempty and~} G\subseteq \mathbb{J}_{max}\cup\{\top\}\}\bigcup \mathbf{Fin}~\!\mathbb{J}_\top$ by (ii).
\item $\mathsf{K}(\Sigma~\!\!\mathbb{J})$ is not a dcpo.

$\mathcal G=\{\mathbb{J}\setminus F : F\in (\mathbb{J}_{max})^{(<\omega)}\}$. Then by (ii), $\mathcal G\subseteq \mathsf{K}(\Sigma~\!\!\mathbb{J}_\top)$ is a filtered family and $\bigcap\mathcal{G}=\bigcap_{F\in (\mathbb{J}_{max})^{(<\omega)}} (\mathbb{J}\setminus F)=\mathbb{J}_{max}\setminus \bigcup (\mathbb{J}_{max})^{(<\omega)}=\emptyset$, whence by Lemma \ref{Kmeet} $\mathcal G$ has no least upper bound in $\mathsf{K}(X)$. Thus $\mathsf{K}(\Sigma~\!\!\mathbb{J})$ is not a dcpo.

\item $\mathsf{K}(\Sigma~\!\!\mathbb{J}_\top)$ is a dcpo.

Suppose that $\{K_d : d\in D\}$ is directed in $\mathsf{K}(\Sigma~\!\!\mathbb{J}_\top)$ (with the Smyth order). Then $\top \in\bigcap_{d\in D}K_d$ and hence $\bigcap_{d\in D}K_d\neq\emptyset$. Now we show that $\bigcap_{d\in D}K_d\in \mathsf{K}(\Sigma~\!\!\mathbb{J}_\top)$. If $\bigcap_{d\in D}K_d=\{\top\}$, then obviously $\bigcap_{d\in D}K_d\in \mathsf{K}(\Sigma~\!\!\mathbb{J}_\top)$. Now we assume $\bigcap_{d\in D}K_d\in \mathsf{K}(\Sigma~\!\!\mathbb{J}_\top)\neq \{\top\}$ and $\{V_i : i\in I\}\subseteq \mathcal \sigma(\mathbb{J}_\top))$ is an open cover of $\bigcap_{d\in D}K_d$. For each $d\in D$ and $i\in I$, let $H_d=K_d\setminus \{\top\}$ and $U_i=V_i\setminus \{\top\}$. Then $H_d\in \mathsf{K}(\Sigma~\!\!\mathbb{J}_\top)$ ($d\in D$), $U_i\in \sigma(\mathbb{J}_\top)$ ($i\in I$) and $\emptyset\neq\bigcap_{d\in D}H_d=\bigcap_{d\in D}K_{d_0}\setminus \{\top\}\subseteq \bigcup_{i\in I}V_i\setminus \{\top\}=\bigcup_{i\in I}U_i$. By \cite[Example 3.1]{LL-2017}, there is $d_0\in D$ such that $H_{d_0}\in \bigcup_{i\in I}U_i$, and consequently, there is $J\in I^{(<\omega)}$ such that $H_d\in \bigcup_{i\in J}U_i$. It follows that $\bigcap_{d\in D}K_d\subseteq K_{d_0}\subseteq \bigcup_{i\in J}V_i$. Thus  $\bigcap_{d\in D}K_d\in \mathsf{K}(\Sigma~\!\!\mathbb{J}_\top)$. By Lemma \ref{Kmeet}, $\mathsf{K}(\Sigma~\!\!\mathbb{J}_\top)$ is a dcpo.

\item $\Sigma~\!\!\mathbb{J}_\top$ is not well-filtered.

Indeed, let $\mathcal K=\{\ua_{\mathbb{J}_\top} (\mathbb{J}_{max}\setminus F) : F\in (\mathbb{J}_{max})^{(<\omega)}\}$. Then by (b), $\mathcal K\subseteq \mathsf{K}(\Sigma~\!\!\mathbb{J}_\top)$ is a filtered family and $\bigcap\mathcal{K}=\bigcap_{F\in (\mathbb{J}_{max})^{(<\omega)}} \ua_{\mathbb{J}_\top} (\mathbb{J}_{max}\setminus F)=\bigcap_{F\in (\mathbb{J}_{max})^{(<\omega)}} ((\mathbb{J}_{max}\setminus F)\cup\{\top\})=\{\top\}\cup(\mathbb{J}_{max}\setminus \bigcup (\mathbb{J}_{max})^{(<\omega)})=\{\top\}\in \sigma (\mathbb{J}_\top)$, but there is no $F\in (\mathbb{J}_{max})^{(<\omega)}$ with $\ua_{\mathbb{J}_\top} (\mathbb{J}_{max}\setminus F)\subseteq \{\top\}$. Therefore, $\Sigma~\!\!\mathbb{J}_\top$ is not well-filtered.

\end{enumerate}
\end{example}

\section{Topological Rudin Lemma, Rudin spaces and well-filtered determined spaces}

In Section 5, we recall some concepts and results about the topological Rudin Lemma, Rudin spaces, $\omega$-Rudin spaces, well-filtered determined spaces and $\omega$-well-filtered determined spaces in \cite{Klause-Heckmann, Shenchon, xu-shen-xi-zhao2, XSXZ-2020, XSXZ-2021} that will be used in
the next four sections.

Rudin's Lemma is a useful tool in non-Hausdorff topology and plays a crucial role in domain theory (see \cite{redbook, quasicont, Heckmann-1992}). Rudin \cite{Rudin} proved her lemma by transfinite methods, using the Axiom of Choice.
Heckmann and Keimel \cite{Klause-Heckmann} presented the following topological variant of Rudin's Lemma.

\begin{lemma}\label{t Rudin} (Topological Rudin Lemma) Let $X$ be a topological space and $\mathcal{A}$ an
irreducible subset of the Smyth power space $P_S(X)$. Then every closed set $C {\subseteq} X$  that
meets all members of $\mathcal{A}$ contains a minimal irreducible closed subset $A$ that still meets all
members of $\mathcal{A}$.
\end{lemma}

Applying Lemma \ref{t Rudin} to the Alexandroff topology on a poset $P$, one obtains the original Rudin's Lemma.

\begin{corollary}\label{rudin} (Rudin's Lemma) Let $P$ be a poset, $C$ a nonempty lower subset of $P$ and $\mathcal F\in \mathbf{Fin} P$ a filtered family with $\mathcal F\subseteq\Diamond C$. Then there exists a directed subset $D$ of $C$ such that $\mathcal F\subseteq \Diamond\da D$.
\end{corollary}

For a $T_0$ space $X$ and $\mathcal{K}\subseteq \mathord{\mathsf{K}}(X)$, let $M(\mathcal{K})=\{A\in \mathcal C(X) : K\bigcap A\neq\emptyset \mbox{~for all~} K\in \mathcal{K}\}$ (that is, $\mathcal K\subseteq \Diamond A$) and $m(\mathcal{K})=\{A\in \mathcal C(X) : A \mbox{~is a minimal member of~} M(\mathcal{K})\}$.

In \cite{Shenchon, XSXZ-2020}, based on topological Rudin's Lemma, Rudin spaces and well-filtered spaces ($\wdd$ spaces for short) were introduced and investigated. These two spaces are closely related to sober spaces and well-filtered spaces (see \cite{Shenchon, XSXZ-2020}).

\begin{definition}\label{rudinset} (\cite{Shenchon, XSXZ-2020})
		Let $X$ be a $T_0$ space.
\begin{enumerate}[\rm (1)]
\item A nonempty subset  $A$  of $X$  is said to have the \emph{Rudin property}, if there exists a filtered family $\mathcal K\subseteq \mathord{\mathsf{K}}(X)$ such that $\overline{A}\in m(\mathcal K)$ (that is,  $\overline{A}$ is a minimal closed set that intersects all members of $\mathcal K$). Let $\mathsf{RD}(X)=\{A\in \mathcal C(X) : A\mbox{~has Rudin property}\}$. The sets in $\mathsf{RD}(X)$ will also be called \emph{Rudin sets}.
    \item $X$ is called a \emph{Rudin space}, $\mathsf{RD}$ \emph{space} for short, if $\ir_c(X)=\mathsf{RD}(X)$, that is, all irreducible closed sets of $X$ are Rudin sets.
        \end{enumerate}
\end{definition}

The Rudin property is called the \emph{compactly filtered property} in \cite{Shenchon}. In order to emphasize its origin from (topological) Rudin's Lemma, such a property was called the Rudin property in \cite{XSXZ-2020}. Clearly, $A$ has Rudin property if{}f $\overline{A}$ has Rudin property (that is, $\overline{A}$ is a Rudin set).

\begin{proposition}\label{rudinwf} (\cite{Shenchon, XSXZ-2020})
	Let $X$ be a $T_0$ space and  $Y$ a well-filtered space. If $f : X\longrightarrow Y$ is continuous and $A\subseteq X$ has Rudin property, then there exists a unique $y_A\in X$ such that $\overline{f(A)}=\overline{\{y_A\}}$.
\end{proposition}

Motivated by Proposition \ref{rudinwf}, the following concept was introduced in \cite{XSXZ-2020}.

\begin{definition}\label{WDspace}(\cite{XSXZ-2020}) Let $X$ be a $T_0$ space.
	\begin{enumerate}[\rm (1)]
\item A subset $A$ of $X$ is called a \emph{well-filtered determined set}, $\wdd$ \emph{set} for short, if for any continuous mapping $ f:X\longrightarrow Y$
to a well-filtered space $Y$, there exists a unique $y_A\in Y$ such that $\overline{f(A)}=\overline{\{y_A\}}$.
Denote by $\mathsf{WD}(X)$ the set of all closed well-filtered determined subsets of $X$.
\item $X$ is called a \emph{well-filtered determined} space, $\mathsf{WD}$ \emph{space} for short, if all irreducible closed subsets of $X$ are well-filtered determined, that is, $\ir_c(X)=\wdd (X)$.
\end{enumerate}
\end{definition}

Obviously, a subset $A$ of a space $X$ is well-filtered determined if{}f $\overline{A}$ is well-filtered determined.

\begin{proposition}\label{SDRWIsetrelation} (\cite{Shenchon, XSXZ-2020})
	Let $X$ be a $T_0$ space. Then $\mathcal S_c(X)\subseteq\mathcal{D}_c(X)\subseteq \mathsf{RD}(X)\subseteq\mathsf{WD}(X)\subseteq\ir_c(X)$.
\end{proposition}

\begin{definition}\label{DCspace} (\cite{XSXZ-2020})
	A $T_0$ space $X$ is called a \emph{directed closure space}, $\mathsf{DC}$ \emph{space} for short, if $\ir_c(X)=\mathcal{D}_c(X)$, that is, for each $A\in \ir_c(X)$, there exists a directed subset of $X$ such that $A=\overline{D}$.
\end{definition}

\begin{corollary}\label{SDRWspacerelation}(\cite[Corollary 6.3]{XSXZ-2020})
	Sober $\Rightarrow$ $\mathsf{DC}$ $\Rightarrow$ $\mathsf{RD}$ $\Rightarrow$ $\mathsf{WD}$.
\end{corollary}

\begin{proposition}\label{WFwdc} (\cite[Corollary 7.11]{XSXZ-2020})
	For a $T_0$ space $X$, the following conditions are equivalent:
	\begin{enumerate}[\rm (1)]
		\item $X$ is well-filtered.
		\item $\mathsf{RD}(X)=\mathcal S_c(X)$.
        \item $\wdd (X)=\mathcal S_c(X)$.
   \end{enumerate}
\end{proposition}

\begin{theorem}\label{soberequiv} (\cite[Theorem 6.6]{XSXZ-2020}) For a $T_0$ space $X$, the following conditions are equivalent:
	\begin{enumerate}[\rm (1)]
		\item $X$ is sober.
		\item $X$ is a $\mathsf{DC}$ $d$-space.
        \item $X$ is a well-filtered $\mathsf{DC}$ space.
		\item $X$ is a well-filtered Rudin space.
		\item $X$ is a well-filtered $\mathsf{WD}$ space.
	\end{enumerate}
\end{theorem}

\begin{proposition}\label{erne1} (\cite[Proposition 3.2]{E_20182}))
	Let $X$ be a locally hypercompact $T_0$ space and $A\in\ir(X)$. Then there exists a directed subset $D\subseteq\da A$ such that $\overline{A}=\overline{D}$. Therefore, $X$ is a $\mathsf{DC}$ space.
\end{proposition}

\begin{proposition}\label{LCrudin and core-compact is WD} (\cite[Theorem 6.10 and Theorem 6.15]{XSXZ-2020}) Let $X$ be a $T_0$ space.
\begin{enumerate}[\rm (1)]
\item If $X$ is locally compact, then $X$ is a Rudin space.
\item If $X$ is core-compact, then $X$ is a $\mathsf{WD}$ space.
\end{enumerate}
\end{proposition}

It is still not known whether every core-compact $T_0$ space is a Rudin space (see \cite[Question 5.14]{xuzhao-20202}).

\begin{question}\label{core-compact imply Smyth power space is WD} For a core-compact $T_0$ space $X$, is the Smyth power space $P_S(X)$ a $\mathsf{WD}$ space? Is the Scott power space $\Sigma~\!\!\mk (X)$ a $\mathsf{WD}$ space?
\end{question}

From Theorem \ref{soberequiv} and Proposition \ref{LCrudin and core-compact is WD} one can immediately get the following result, which was first proved by Lawson, Wu and Xi \cite{Lawson-Xi} using a different method.

\begin{corollary}\label{WFcorcomp-sober} (\cite{Lawson-Xi, XSXZ-2020}) Every core-compact well-filtered space is sober.
\end{corollary}

By Corollary \ref{WFcorcomp-sober}, Theorem \ref{SoberLC=CoreC} can be strengthened into the following one.

\begin{theorem}\label{SoberLC=CoreCNew}  For a $T_0$ space $X$, the following conditions are equivalent:
\begin{enumerate}[\rm (1)]
	\item $X$ locally compact and sober.
	\item $X$ is locally compact and well-filtered.
	\item $X$ is core-compact and sober.
    \item $X$ is core-compact and well-filtered.
\end{enumerate}
\end{theorem}

Figure 3 shows certain relations among some kinds of spaces.

\begin{figure}[ht]
	\centering
	\includegraphics[height=1.6in,width=4.0in]{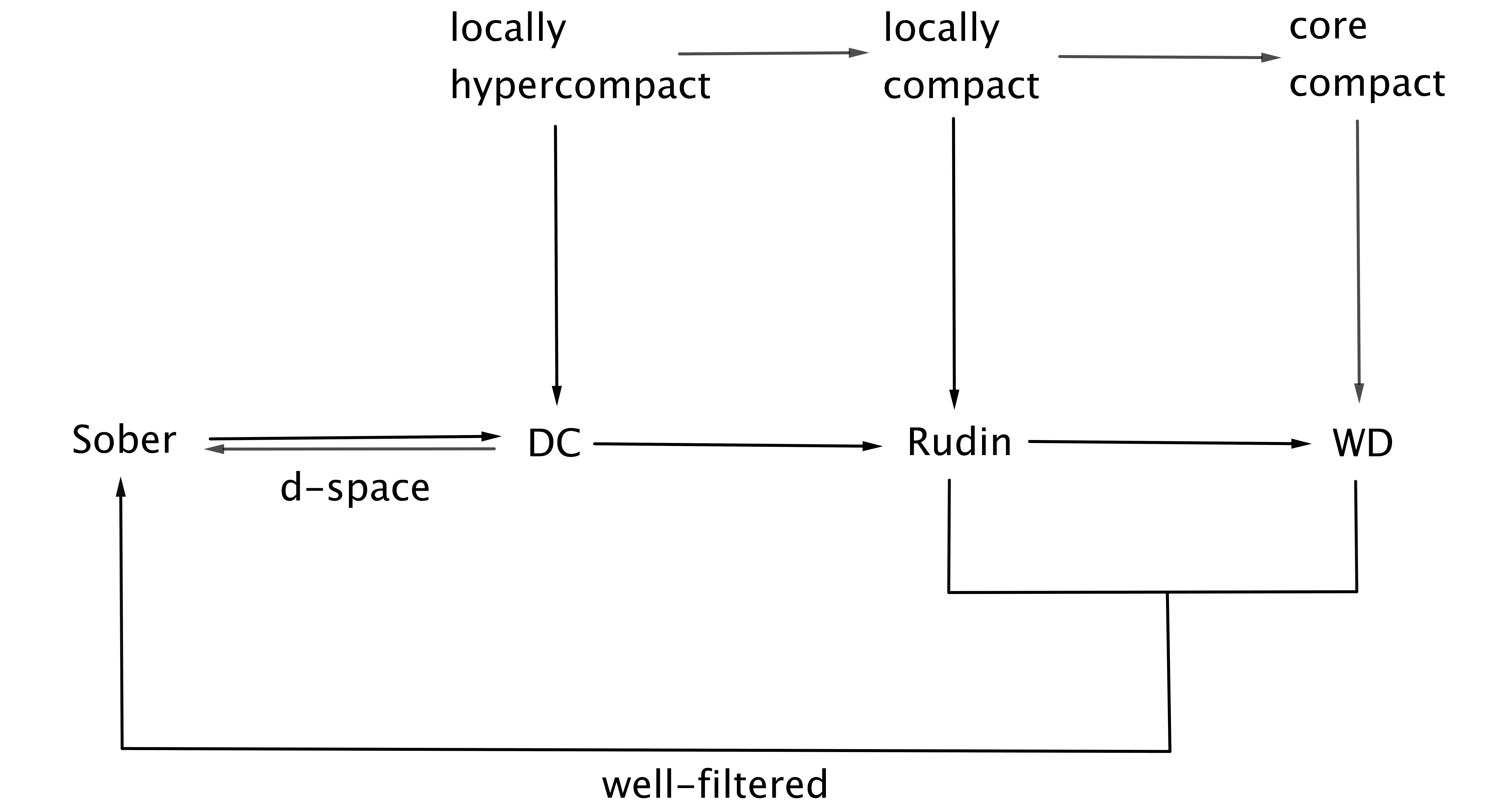}
	\caption{Certain relations among some kinds of spaces}
\end{figure}

In order to emphasize the Scott topology, we introduce the following notions.

\begin{definition}\label{sober dcpo} A poset $P$ is called a \emph{sober dcpo} (resp., a \emph{well-filtered dcpo}) if $\Sigma~\!\! P$ is a sober space (resp., well-filtered space).
\end{definition}

Clearly, a sober dcpo is a well-filtered dcpo. For Isblle's lattice $L$ constructed in \cite{isbell}, $\Sigma~\!\! L$ is non-sober, namely, $L$ is not a  sober dcpo, and by Corollary \ref{Xi-Lawson result 1}, $L$ is well-filtered. The Johnstone's dcpo $\mathbb{J}$ (see Example \ref{K(X) is dcpo not implies X is wf}) is not well-filtered.

\begin{definition}\label{Rudin dcpo} Let $P$ be a poset.
\begin{enumerate}[\rm (1)]
\item $P$ is said to be a $\mathsf{DC}$ \emph{poset} if $\Sigma~\!\!P$ is a $\mathsf{DC}$ space.
\item $P$ is said to be a \emph{Rudin poset} if $\Sigma~\!\! P$ is a Rudin space.
    \item  $P$ is said to be a \emph{well-filtered determined poset}, a $\wdd$-\emph{poset} for short, if $\Sigma~\!\! P$ is a well-filtered determined space.
        \item When a dcpo $P$ is a Rudin poset (resp., a well-filtered determined poset), we will call $P$ a \emph{Rudin dcpo} (resp., a \emph{well-filtered determined dcpo}).
    \end{enumerate}
\end{definition}

The following corollary follows directly from Theorem \ref{soberequiv}.

\begin{corollary}\label{sober dcpo equivalent} For a poset $P$, the following conditions are equivalent:
\begin{enumerate}[\rm (1)]
\item $P$ is a sober dcpo.
\item $P$ is a $\mathsf{DC}$ dcpo.
\item $P$ is a $\mathsf{DC}$ well-filtered dcpo.
\item $P$ is a Rudin well-filtered dcpo.
\item $P$ is a $\wdd$ well-filtered dcpo.
\end{enumerate}
\end{corollary}

In \cite{xu-shen-xi-zhao2}, the following countable versions of Rudin spaces and $\mathsf{WD}$ spaces were introduced and studied.

\begin{definition}\label{omega rudinset}  (\cite[Definition 5.1]{xu-shen-xi-zhao2})  Let $X$ be a $T_0$ space and $A$ a nonempty subset of $X$.
\begin{enumerate}[\rm (a)]
\item The set $A$ is said to be an $\omega$-\emph {Rudin set}, if there exists a countable filtered family $\mathcal K\subseteq \mathord{\mathsf{K}}(X)$ such that $\overline{A}\in m(\mathcal K)$. Let $\mathsf{RD}_\omega(X)$ denote the set of all closed $\omega$-Rudin sets of $X$.
    \item The space $X$ is called $\omega$-\emph{Rudin space}, if $\ir_c(X)=\mathsf{RD}_\omega(X)$ or, equivalently, all irreducible (closed) subsets of $X$ are $\omega$-Rudin sets.
        \end{enumerate}
\end{definition}

\begin{definition}\label{omega-WF space} (\cite[Definition 3.9]{xu-shen-xi-zhao2})
	A $T_0$ space $X$ is called \emph{$\omega$-well-filtered}, if for any countable filtered family $\{K_n : n<\omega\}\subseteq \mk (X)$ and  $U\in\mathcal O(X)$, it holds that
		$$\bigcap_{n<\omega}K_n\subseteq U \ \Rightarrow \ \exists n_0<\omega, K_{n_0}\subseteq U.$$
\end{definition}

\begin{definition}\label{omega WDspace} (\cite[Definition 5.4]{xu-shen-xi-zhao2})   Let $X$ be a $T_0$ space and $A$ a nonempty subset of $X$.
\begin{enumerate}[\rm (a)]
\item The set $A$ is called an $\omega$-\emph{well}-\emph{filtered determined set}, $\omega$-$\wdd$ \emph{set} for short, if for any continuous mapping $ f:X\longrightarrow Y$
to an $\omega$-well-filtered space $Y$, there exists a (unique) $y_A\in Y$ such that $\overline{f(A)}=\overline{\{y_A\}}$.
Denote by $\mathsf{WD}_\omega(X)$ the set of all closed $\omega$-well-filtered determined subsets of $X$.
\item The space $X$ is called $\omega$-\emph{well}-\emph{filtered determined}, $\omega$-$\mathsf{WD}$ \emph{space} for short, if $\ir_c(X)=\mathsf{WD}_\omega(X)$ or, equivalently, all irreducible (closed) subsets of $X$ are $\omega$-well-filtered determined.
 \end{enumerate}
\end{definition}

For a $T_0$ space $X$, it was proved in \cite[Proposition 5.5]{xu-shen-xi-zhao2} that $\mathcal{S}_c(X)\subseteq \mathsf{RD}_\omega(X)\subseteq\mathsf{WD}_\omega(X)\subseteq\ir_c(X)$. Therefore, every $\omega$-Rudin space is $\omega$-well-filtered determined.

By \cite[Theorem 5.11]{xu-shen-xi-zhao2}, we have the following similar result to Theorem \ref{soberequiv}.

\begin{proposition}\label{sober equiv using omega RD and omega WD}  For a $T_0$ space $X$, the following conditions are equivalent:
	\begin{enumerate}[\rm (1)]
		\item $X$ is sober.
		\item $X$ is an $\omega$-Rudin and $\omega$-well-filtered space.
		\item $X$ is an $\omega$-well-filtered determined and $\omega$-well-filtered space.
	\end{enumerate}
\end{proposition}

\begin{theorem} \label{sobrifcaltion first-countable is Rudin} (\cite[Theorem 5.6 and Theorem 6.12]{XSXZ-2021}) Let $X$ be a $T_0$ space.
\begin{enumerate}[\rm (1)]
\item If the sobrification $X^s$ of $X$ is first-countable, then $X$ is an $\omega$-Rudin space.
\item If $X$ is first-countable, then $X$ is a $\mathsf{WD}$ space.
\end{enumerate}
\end{theorem}

From Theorem \ref{soberequiv} and Theorem \ref{sobrifcaltion first-countable is Rudin} we immediately deduce the following result.

\begin{corollary}\label{first-countable WF is sober} (\cite[Theorem 4.2]{xu-shen-xi-zhao2}) Every first-countable well-filtered $T_0$ space is sober.
\end{corollary}

It is still not known whether a first-countable $T_0$ space is a Rudin space (see \cite[Problem 6.15]{XSXZ-2021}). Since the first-countability is a hereditary property, from Remark \ref{xi embdding} and Theorem \ref{sobrifcaltion first-countable is Rudin} we know that if the Smyth power space $P_S(X)$ of a $T_0$ space $X$ is first-countable, then $X$ is a $\mathsf{WD}$ space.

So naturally we ask the following question.

\begin{question}\label{Smith power space CI X is RD} Is a $T_0$ space with a first-countable Smyth power space a Rudin space?
\end{question}

In Example \ref{Scott sober not implies X is sober} a $T_0$ space $X$ is given for which the Scott power space $\Sigma~\!\!\mathsf{K}(X)$ is a first-countable sober $c$-space but $X$ is not a $\mathsf{WD}$ space (and hence not a Rudin space).

By Proposition \ref{sober equiv using omega RD and omega WD} and Theorem \ref{sobrifcaltion first-countable is Rudin}, we have the following result.

\begin{corollary}\label{sobrification first-countable and omega-WF is sober}(\cite[Theorem 5.9]{XSXZ-2021}) Every $\omega$-well-filtered space with a first-countable sobrification is sober.
\end{corollary}

 In Theorem \ref{sobrifcaltion first-countable is Rudin} and Corollary \ref{sobrification first-countable and omega-WF is sober}, the first-countability of $X^s$ can not be weakened to that of $X$ as shown in the following example. It is also shows that the first-countability of a $T_0$ space $X$ does not imply the first-countability of $X^s$ in general.

\begin{example}\label{first-countable omega WF is not sober}  Let $\omega_1$ be the first uncountable ordinal number and $P=[0, \omega_1)$. Then
\begin{enumerate}[\rm (a)]
\item $\mathcal C(\Sigma P)=\{\da t : t\in P\}\cup\{\emptyset, P\}$.
\item $\Sigma P$ is first-countable and compact (since $P$ has a least element $0$).

\item $(\Sigma P)^s$ is not first-countable.

In fact, it is easy to verify that $(\Sigma P)^s$ is homeomorphic to $\Sigma [0, \omega_1]$. Since sup of a countable family of countable ordinal numbers is still a countable ordinal number, $\Sigma [0, \omega_1]$ has no countable base at the point $\omega_1$.

\item $\mathsf{K}(\Sigma P)=\{\ua x : x\in P\}$ and $\Sigma P$ is not an $\omega$-Rudin space.

    For $K\in \mathsf{K}(\Sigma P)$, we have $\mathrm{inf}~K\in K$, and hence $K=\ua \mathrm{inf}~K$. So $\mathsf{K}(\Sigma P)=\{\ua x : x\in P\}$. Now we show that the irreducible closed set $P$ is not an $\omega$-Rudin set. For any countable filtered family $\{\ua \alpha_n : n\in\mn \}\subseteq \mathsf{K}(\Sigma P)$, let $\beta=\mathrm{sup}\{\alpha_n : n\in\mn\}$. Then $\beta$ is still a countable ordinal number. Clearly, $\da \beta \in M(\{\ua \alpha_n : n\in \mn\})$ and $P\neq \da \beta$. Therefore, $P\notin m(\{\ua \alpha_n : n\in \mn\})$. Thus $P$ is not an $\omega$-Rudin set, and hence $\Sigma P$ is not an $\omega$-Rudin space.

    \item $\Sigma P$ is a Rudin space.

    It is easy to check that $\ir_c(\Sigma P)=\{\downarrow x : x\in P\}\cup \{P\}$. Clearly, $\downarrow x$ is a Rudin set for each $x\in P$. Now we show that $P$ is a Rudin set. First, $\{\ua s : s\in P\}$ is filtered. Second, $P\in M(\{\ua s : s\in P\})$. For a closed subset $B$ of $\Sigma P$, if $B\neq P$, then $B=\da t$ for some $t\in P$, and hence $\ua (t+1)\cap\da t=\emptyset$. Thus $B\notin M(\{\ua s : s\in P\})$, proving that $P$ is a Rudin set.

     \item $P$ is not a dcpo (note that $P$ is directed and $\vee P$ does not exist). So $\Sigma P$ is not a $d$-space, and hence $\Sigma P$ is neither well-filtered nor sober.

    \item $\Sigma P$ is $\omega$-well-filtered.

    If $\{\ua x_n : n\in \mn\}\subseteq \mathsf{K}(\Sigma P)$ is countable filtered family and $U\in \sigma (P)$ with $\bigcap_{n\in \mn}\ua x_n\subseteq U$, then $\{x_n : i\in\mn\}$ is a countable subset of $P=[0, \omega_1)$. Since sup of a countable family of countable ordinal numbers is still a countable ordinal number, we have $\beta=\mathrm{sup}\{x_n : n\in \mn\}\in P$, and hence $\ua \beta=\bigcap_{n\in \mn}\ua x_n\subseteq U$. Therefore, $\beta\in U$, and consequently, $x_n\in U$ for some $n\in \mn$ or, equivalently, $\ua x_n\subseteq U$, proving that $\Sigma P$ is $\omega$-well-filtered.
\end{enumerate}
\end{example}

\section{Well-filteredness of Scott power spaces}

 In this section, we mainly discuss the following two questions:
\vskip 0.1cm
 \textbf{Question 1.} Is the Scott power space $\Sigma~\!\! \mk(X)$ of a $d$-space $X$ a $d$-space?
\vskip 0.1cm
 \textbf{Question 2.} Is the Scott power space $\Sigma~\!\! \mk(X)$ of a well-filtered space $X$ well-filtered?
\vskip 0.1cm

First, Example \ref{X d-space not imply Scott power space d-space} below shows that there is a second-countable Noetherian $d$-space $X$ for which $\mk (X)$ is not a dcpo and hence neither the Smyth power space $P_S(X)$ nor the Scott power space $\Sigma~\!\!\mk (X)$ is a $d$-space, which gives a negative answer to Question 1.

In order to present the example, we need the following lemma.

\begin{lemma}\label{local comp imply V topol finer than Scott topol} (\cite[Lemma 7.26]{Schalk}) For a locally compact $T_0$ space $X$, the Scott topology is coarser than the upper Vietoris topology on $\mathsf{K}(X)$, that is, $\sigma(\mathsf{K}(X))\subseteq \mathcal O(P_S(X))$.
\end{lemma}

\begin{proof} It was proved by Schalk in \cite{Schalk} (see \cite[the proof of Lemma 7.26]{Schalk}). We present a more direct proof here.

Suppose that $\mathcal U\in \sigma(\mathsf{K}(X))$ and $K\in \mathcal U$. Let $\mathcal K=\{G\in \mathsf{K}(X) : K\subseteq \ii ~\!G\}$. Now we show that $\mathcal K$ is filtered and $G=\bigcap \mathcal K$.

{$1^{\circ}$} For each $U\in \mathcal O(X)$ with $K\subseteq U$, there is $G_U\in \mathcal K$ with $G_U\subseteq U$.

If $U\in \mathcal O(X)$ for which $K\subseteq U$, then for each $x\in K$, there is $K_x\in \mathsf{K}(X)$ such that $x\in \ii~\!K_x\subseteq K_x\subseteq U$  since $X$ is locally compact. By the compactness of $K$, there is $\{x_1, x_2, ..., x_n\}\subseteq K$ such that $K\subseteq \bigcup\limits_{i=1}^{n} \ii~\!K_{x_i}$. Let $G_U=\bigcup\limits_{i=1}^{n} K_{x_i}$. Then $K\subseteq \ii~\!G_U\subseteq G_U\subseteq U$, whence $G_U\in \mathcal K$ and $G_U\subseteq U$.

{$2^{\circ}$} $\mathcal K$ is filtered.

Suppose that $G_1, G_2\in \mathcal K$. Then $K\subseteq \ii~\!G_1\cap \ii~\!G_2$. Hence, by what was shown above, there is $G_3\in \mathcal K$ with $G_3\subseteq \ii~\!G_1\cap \ii~\!G_2\subseteq G_1\cap G_2$, proving the filteredness of $\mathcal K$.

By $1^{\circ}$ and $2^{\circ}$, $K\subseteq \bigcap \mathcal K\subseteq \{U\in \mathcal O(X) : K\subseteq U\}=K$, whence $K=\bigcap \mathcal K=\bigvee_{\mathsf{K}(X)}\mathcal K$ by Lemma \ref{Kmeet}. Since $K\in \mathcal U\in \sigma(\mathsf{K}(X))$, $G\in \mathcal U$ for some $G\in \mathcal K$. Hence $K\in \Box_{\mathsf{K}(X)}\ii~\!G\subseteq \mathcal U$. Thus $\mathcal U\in \mathcal O(P_S(X))$.

\end{proof}

By Corollary \ref{wf space Vietoris less Scott} and Lemma \ref{local comp imply V topol finer than Scott topol}, we get the following corollary.

\begin{corollary}\label{LC sober domain V=S} (\cite[Lemma 7.26]{Schalk}) If $X$ is a locally compact sober space (equivalently, a locally compact well-filtered space or a core-compact well-filtered space), then the upper Vietoris topology and the Scott topology on $\mk (X)$ coincide.
\end{corollary}

Considering Remark \ref{core-compact is not locally compact} and Lemma \ref{local comp imply V topol finer than Scott topol}, we have the following question.

\begin{question}\label{core-compact imply V topol finer than Scott topol} For a core-compact $T_0$ space $X$, is the Scott topology coarser than the upper Vietoris topology on $\mathsf{K}(X)$?
\end{question}

\begin{example}\label{X d-space not imply Scott power space d-space}
	Let $X$ be a countably infinite set (for example, $X=\mathbb{N}$) and $X_{cof}$ the space equipped with the \emph{co-finite topology} (the empty set and the complements of finite subsets of $X$ are open). Then
\begin{enumerate}[\rm (a)]
    \item $\mathcal C(X_{cof})=\{\emptyset, X\}\cup X^{(<\omega)}$, $X_{cof}$ is $T_1$ and hence a $d$-space.
    \item $\ir_c (X_{cof})=\{\{x\} : x\in X\}\cup \{X\}$.
    \item $\mk (X_{cof})=2^X\setminus \{\emptyset\}$.
    \item $X_{cof}$ is second-countable.

    Clearly, $\mathcal O(X_{cof})$ is countable, and hence $X_{cof}$ is second-countable.
    \item $X_{cof}$ is Noetherian and hence locally compact.

    Since every subset of $X$ is compact in $X_{cof}$, the space $X_{cof}$ is a Noetherian space and hence a locally compact space.
    \item $X_{cof}$ is a Rudin space.

    By (e) and Proposition \ref{LCrudin and core-compact is WD} (or by (d) and Corollary \ref{second-countable is omega Rudin} below), $X_{cof}$ is a Rudin space.

 \item $\mk (X_{cof})$ is not a dcpo and hence $X_{cof}$ is neither well-filtered nor sober.

    $\mathcal{K}=\{X\setminus F : F\in X^{(<\omega)}\}\subseteq\mk (X_{cof})$ is countable filtered and $\bigcap \mathcal{K}_X=X\setminus \bigcup X^{(<\omega)}=X\setminus X=\emptyset$, whence $\bigvee \mathcal K$ does not exist in $\mk (X_{cof})$ by Lemma \ref{Kmeet}. Thus $\mk (X_{cof})$ is not a dcpo, whence by Remark \ref{sober implies WF implies d-space} and Theorem \ref{Smythwf}, $X_{cof}$ is neither well-filtered nor sober.

    \item The upper Vietoris topology and the Scott topology on $\mk (X_{cof})$ agree.

    By the local compactness of $X_{cof}$ and Lemma \ref{local comp imply V topol finer than Scott topol}, we have $\sigma (\mk (X_{cof})\subseteq \mathcal O(P_S(X_{cof}))$. Now we show that $\Box U\in \sigma (\mathsf{K}(X_{cof})$
    for each $U\in \mathcal O(X_{cof})\setminus \{\emptyset\}$. Clearly, $\Box U=\ua_{\mathsf{K}(X_{cof})}\Box U$. Suppose that $\mathcal K_D=\{K_d : d\in D\}\in \mathcal D(\mk (X_{cof})$ and $\bigvee_{\mk (X_{cof})} \mathcal K_D\in \Box U$. Then by Lemma \ref{Kmeet} $\bigcap_{d\in D}K_d=\bigvee_{\mk (X_{cof})} \mathcal K_D\subseteq U$ or, equivalently, $X\setminus U\subseteq \bigcup_{d\in D}(X\setminus K_d)$. Since $X\setminus U$ is finite and $ \{X\setminus K_d : d\in D\}$ is directed, there is $d_0\in D$ with $X\setminus U\subseteq X\setminus K_{d_0}$, whence $K_{d_0}\setminus U$. Thus $\Box U\in \sigma (\mathsf{K}(X_{cof})$. Therefore, $\mathcal O(P_S(X_{cof}))\subseteq\sigma (\mk (X_{cof})$ and hence $\sigma (\mk (X_{cof})=\mathcal O(P_S(X_{cof}))$.

 \item $\Sigma~\!\!\mk (X_{cof})$ is not a $d$-space and hence it is neither a well-filtered space nor a sober space.

    Since $\mk (X_{cof})$ is not a dcpo, $\Sigma~\!\!\mk (X_{cof})$ is not a $d$-space. By Remark \ref{sober implies WF implies d-space}, $\Sigma~\!\!\mk (X_{cof})$ is neither a well-filtered space nor a sober space.
 \end{enumerate}
\end{example}

Now we investigate Question 2. First, as one of the main results of this paper, we have the following conclusion.

\begin{theorem}\label{wf imply Scott wf}
	For a well-filtered space $X$, $\Sigma~\!\! \mathsf{K}(X)$ is well-filtered.
\end{theorem}
\begin{proof} By Corollary \ref{wf space Vietoris less Scott}, $\mk (X)$ is a dcpo and $\mathcal O(P_S(X))\subseteq \sigma (\mk (X))$ (i.e., $\Box U\in \sigma (\mk (X))$ for all $U\in O(X)$). Suppose that $\{\mathcal K_d : d\in D\}\subseteq \mk(\Sigma~\!\! \mk (X))$ is filtered, $\mathcal U\in \sigma (\mk (X))$ and $\bigcap\limits_{d\in D} \mathcal K_d \subseteq \mathcal U$. If $\mathcal K_d\not\subseteq \mathcal U$ for each $d\in D$, that is, $\mathcal K_d\bigcap (\mk (X)\setminus \mathcal U)\neq\emptyset$, then $\{\mathcal K_d : d\in D\}\in \ir (P_S(\mk(\Sigma~\!\! \mk (X))))$ and hence by Lemma \ref{t Rudin} $\mk (X)\setminus \mathcal U$ contains a minimal irreducible closed subset $\mathcal A$ that still meets all
members $\mathcal K_d$. For each $d\in D$, let $K_d=\bigcup \ua_{\mathsf{K}(X)} (\mathcal K_d\cap \mathcal A)$.

{\bf Claim 1:} For each $d\in D$, $K_d\in \mk (X)$ and $K_d\in \mathcal A$.

By $\mathcal K_d\in \mk(\Sigma~\!\! \mk (X))$ and $\mathcal A\in \mathcal C(\Sigma~\!\! \mk (X))$, we have that $\ua_{\mathsf{K}(X)} (\mathcal K_d\cap \mathcal A)\in \mk(\Sigma~\!\! \mk (X))$, and hence $\ua_{\mathsf{K}(X)} (\mathcal K_d\cap \mathcal A)\in \mk (P_s(X))$ by $\mathcal O(P_S(X))\subseteq \sigma (\mk (X))$. By Lemma \ref{K union}, $K_d=\bigcup \ua_{\mathsf{K}(X)} (\mathcal K_d\cap \mathcal A)=\bigcup (\mathcal K_d\cap \mathcal A)\in \mk (X)$. Since $\mathcal A=\da_{\mk (X)}\mathcal A$ and $\mathcal K_d\cap \mathcal A\neq\emptyset$, we have $K_d\in\mathcal A$.

{\bf Claim 2:} $\{K_d : d\in D\}\subseteq \mk (X)$ is filtered (by Claim 1 and the filteredness of $\{\mathcal K_d : d\in D\}$).

{\bf Claim 3:} $K=\bigcap_{d\in D}K_d\in \mk (X)$ and $K\in \mathcal A$.

By the well-filteredness of $X$, $K=\bigcap_{d\in D}K_d\in \mk (X)$. By Claim 1, Claim 2 and Lemma \ref{Kmeet}, $K=\bigvee_{\mk (X)} \{K_d : d\in D\}\in \mathcal A$ since $\mathcal A\in \mathcal C(\sigma (\mk (X)))$.

{\bf Claim 4:} For each $k\in K$, $\mathcal A\subseteq \Diamond_{\mk (X)}\overline{\{k\}}$.

For each $d\in D$, we have $k\in K\subseteq K_d=\bigcup (\mathcal K_d\cap \mathcal A)$, whence there is $G_d\in \mathcal K_d\cap \mathcal A$ such that $k\in G_d$, and consequently, $G_d\in \mathcal K_d\cap \mathcal A\cap \Diamond_{\mk (K)}\overline{\{k\}}$. Therefore, $\mathcal A\cap \Diamond_{\mk (K)}\overline{\{k\}}\in M(\{\mathcal K_d : d\in D\})$. By the minimality of $\mathcal A$ and $\Diamond_{\mk (K)}\overline{\{k\}}\in \mathcal C(P_S(X))\subseteq \mathcal C(\Sigma~\!\!\mk (X))$, we have $\mathcal A=\Diamond_{\mk (K)}\overline{\{k\}}\bigcap\mathcal A$, that is, $A\subseteq \Diamond_{\mk (K)}\overline{\{k\}}$.

{\bf Claim 5:} $A=\da_{\mk (X)}K$.

By Claim 3 and Claim 4, $\da_{\mk (X)}K\subseteq\mathcal A\subseteq \bigcap_{k\in K} \Diamond_{\mk (K)}\overline{\{k\}}$. Clearly,
$$\begin{array}{lll}
	G\in \bigcap_{k\in K} \Diamond_{\mk (X)}\overline{\{k\}} &\Leftrightarrow& \forall k\in K, G \in \Diamond_{\mk (K)}\overline{\{k\}}\\
	&\Leftrightarrow&\forall k\in K, G\cap \overline{\{k\}}\neq\emptyset\\
	&\Leftrightarrow& \forall k\in K, k\in G\\
	&\Leftrightarrow& K\subseteq G.
	\end{array}$$
	This implies that $\bigcap_{k\in K} \Diamond_{\mk (K)}\overline{\{k\}}=\da_{\mk (X)}K$, and hence $A=\da_{\mk (X)}K$.

{\bf Claim 5:} $K\in \bigcap_{d\in D}\mathcal K_d$.

For each $d\in D$, by $\mathcal K_d\bigcap \mathcal A\neq\emptyset$, $\mathcal K_d=\ua_{\mk (X)}\mathcal K_d$ and $A=\da_{\mk (X)}K$, we have $K\in \mathcal K_d$, whence $K\in \bigcap_{d\in D}\mathcal K_d\subseteq \mathcal U$, being a contradiction with $K\in \mathcal A\subseteq \mk (X)\setminus \mathcal U$.

Therefore, there is $d_0\in D$ such that $\mathcal K_{d_0}\subseteq\mathcal U$, proving that $\Sigma~\!\! \mk (X)$ is well-filtered.

\end{proof}

Example \ref{Scott sober not implies X is wf} shows that unlike Smyth power spaces (see Theorem \ref{Smythwf}), the converse of Theorem \ref{wf imply Scott wf} does not hold.

From Theorem \ref{SoberLC=CoreCNew} and Theorem \ref{wf imply Scott wf} we deduce the following result.

\begin{corollary}\label{wf Scott LC=core compact} For a well-filtered space $X$, the following two conditions are equivalent:
\begin{enumerate}[\rm (1)]
\item $\Sigma~\!\!\mk (X)$ is core-compact.
\item  $\Sigma~\!\!\mk (X)$ is locally compact.
\end{enumerate}
\end{corollary}

By Theorem \ref{soberequiv}, Theorem \ref{SoberLC=CoreCNew} and Theorem \ref{wf imply Scott wf}, we have the following two corollaries.

\begin{corollary}\label{wf Scott power space sober equi} For a well-filtered space $X$, the following three conditions are equivalent:
\begin{enumerate}[\rm (1)]
\item $\Sigma~\!\!\mk (X)$ is sober.
\item $\Sigma~\!\!\mk (X)$ is Rudin.
\item $\Sigma~\!\!\mk (X)$ is well-filtered determined.
\end{enumerate}
\end{corollary}

\begin{corollary}\label{K union in Scott power space}  Let $X$ be a well-filtered space.
\begin{enumerate}[\rm (1)]
\item If $\mathcal K\in\mk(\Sigma~\!\!\mk (X))$, then $\bigcup \mathcal K\in\mk(X)$.
\item The mapping $\bigcup : \Sigma~\!\!\mk(\Sigma~\!\!\mk (X)) \longrightarrow \Sigma~\!\!\mk (X)$, $\mathcal K\mapsto \bigcup \mathcal K$, is continuous.
\end{enumerate}
\end{corollary}
\begin{proof} (1): By Corollary \ref{wf space Vietoris less Scott}, $\mathcal O(P_S(X))\subseteq \sigma (\mk (X))$. For $\mathcal K\in \mk(\Sigma~\!\!\mk (X))$, we have $\mathcal K\in \mk(P_S(X))$ since $\mathcal O(P_S(X))\subseteq \sigma (\mk (X))$. Then by Lemma \ref{K union}, $\bigcup \mathcal K\in\mk(X)$.

(2): Suppose that $\{\mathcal K_d : d\in D\}\subseteq \mk(\Sigma~\!\!\mk (X))$ is directed (with the Smyth order) for which $\bigvee_{\mk(\Sigma~\!\!\mk (X))}\{\mathcal K_d : d\in D\}$ exists. Then by Lemma \ref{Kmeet}, $\bigvee_{\mk(\Sigma~\!\!\mk (X))}\{\mathcal K_d : d\in D\}=\bigcap_{d\in D}\mathcal K_d$. It follows that $\bigcup \bigvee_{\mk(\Sigma~\!\!\mk (X))}\{\mathcal K_d : d\in D\}=\bigcup \bigcap_{d\in D}\mathcal K_d$ and $\bigvee_{d\in D} \bigcup \mathcal K_d=\bigcap_{d\in D}\bigcup \mathcal K_d=\bigcup_{\varphi\in \prod_{d\in D}\mathcal K_d}\bigcap_{d\in D}\varphi(d)$ by Lemma \ref{Kmeet}. For each $K\in \bigcap_{d\in D}\mathcal K_d$, define $\varphi_{K}\in \prod_{d\in D}\mathcal K_d$ by $\varphi_K(d)\equiv K$ for all $d\in D$. Then $\bigcap_{d\in D}\varphi_K(d)=K$. Hence $\bigcup \bigcap_{d\in D}\mathcal K_d\subseteq \bigcup_{\varphi\in \prod_{d\in D}\mathcal K_d}\bigcap_{d\in D}\varphi(d)$.

Conversely, for each $\varphi\in \prod_{d\in D}\mathcal K_d$, $x\in \bigcap_{d\in D}\varphi(d)$ and $d^{\prime}\in D$, we have that $\ua x\sqsupseteq\varphi (d^{\prime})\in \mathcal K_{d^{\prime}}=\ua_{\mk(\Sigma~\!\!\mk (X))}\mathcal K_d^{\prime}$, and consequently, $\ua x\in \bigcap_{d^{\prime}\in D}\mathcal K_d^{\prime}$ and hence $\ua x\subseteq \bigcup \bigcap_{d\in D}\mathcal K_d$. It follows that $\bigcap_{d\in D}\varphi(d)\subseteq \bigcup \bigcap_{d\in D}\mathcal K_d$. Therefore, $\bigcup_{\varphi\in \prod_{d\in D}\mathcal K_d}\bigcap_{d\in D}\varphi(d)\subseteq  \bigcup \bigcap_{d\in D}\mathcal K_d$.

Thus $\bigcup \bigvee_{\mk(\Sigma~\!\!\mk (X))}\{\mathcal K_d : d\in D\}=\bigvee_{d\in D} \bigcup \mathcal K_d$. By Lemma \ref{Scott continuous equiv}, $\bigcup : \Sigma~\!\!\mk(\Sigma~\!\!\mk (X)) \longrightarrow \Sigma~\!\!\mk (X)$ is continuous.
\end{proof}

\begin{proposition}\label{Scott power WF implies X is also}  Let $X$ be a $T_0$ space. If the upper Vietoris topology is coarser than the Scott topology on $\mathsf{K}(X)$ (that is, $\mathcal O(P_S(X))\subseteq \sigma (\mathsf{K}(X))$), and $\Sigma~\!\! \mathsf{K}(X)$ is well-filtered, then $X$ is well-filtered.
\end{proposition}

\begin{proof} Suppose that $\{K_d : d\in D\}\subseteq \mathsf{K}(X)$ is filtered and $U\in \mathcal O(X)$ with $\bigcap_{d\in D}K_d\subseteq U$. Then $\{\ua_{\mathsf{K}(X)}K_d : d\in D\}\subseteq \mathsf{K}(\Sigma~\!\! \mathsf{K}(X))$ is filtered, $\Box U\in \mathcal O(P_S(X))\subseteq \sigma (\mathsf{K}(X))$ and $\bigcap_{d\in D}\ua_{\mathsf{K}(X)}K_d\subseteq \Box U$. By the well-filteredness of $\Sigma~\!\! \mathsf{K}(X)$, there is $d\in D$ such that  $\ua_{\mathsf{K}(X)}K_d\subseteq \Box U$, and hence $K_d\subseteq U$. Thus $X$ is well-filtered.
 \end{proof}

Example \ref{Scott sober not implies X is wf} below shows that when $X$ lacks
the condition of $\mathcal O(P_S(X))\subseteq \sigma (\mathsf{K}(X))$, Proposition \ref{Scott power WF implies X is also} may not hold.

\begin{corollary}\label{X WF Scott power WF equivalent}  For a $T_0$ space $X$, the following conditions are equivalent:
\begin{enumerate}[\rm (1)]
\item $X$ is well-filtered.
\item The upper Vietoris topology is coarser than the Scott topology on $\mathsf{K}(X)$, and $\Sigma~\!\! \mathsf{K}(X)$ is well-filtered.
\item The upper Vietoris topology is coarser than the Scott topology on $\mathsf{K}(X)$, and $\Sigma~\!\! \mathsf{K}(X)$ is a $d$-space.
\item $\mathsf{K}(X)$ is a dcpo, and the upper Vietoris topology is coarser than the Scott topology on $\mathsf{K}(X)$.
\end{enumerate}
\end{corollary}
\begin{proof} (1) $\Rightarrow$ (2): By Corollary \ref{wf space Vietoris less Scott} and Theorem \ref{wf imply Scott wf}.

(2) $\Rightarrow$ (3): By Remark \ref{sober implies WF implies d-space}.

(3) $\Rightarrow$ (4): Trivial.

(4) $\Rightarrow$ (1): By (4), $P_S(X)$ is a $d$-space, whence $X$ is well-filtered by Theorem \ref{Smythwf}.
\end{proof}

\section{Non-sobriety of Scott power space of a sober space}

 In this section, we investigate the following question:
\vskip 0.1cm
 \textbf{Question 3.} Is the Scott power space $\Sigma~\!\! \mk(X)$ of a sober space $X$ sober?
 \vskip 0.1cm

First, the following example shows that there is a well-filtered space $X$ for which its Scott power space $\Sigma~\!\!\mk (X)$ is a first-countable sober $c$-space, but $X$ is not sober although its Scott power space $\Sigma~\!\! \mk (X)$ is sober by Corollary \ref{wf Scott CI imply sober}. Hence, by Corollary \ref{WFcorcomp-sober} and Corollary \ref{first-countable WF is sober}, $X$ is neither core-compact nor first-countable. So the sobriety of the Scott power space of a $T_0$ space $X$ does not imply the sobriety of $X$ in general.

\begin{example}\label{Scott sober not implies X is sober}
	Let $X$ be an uncountably infinite set and $X_{coc}$ the space equipped with \emph{the co-countable topology} (the empty set and the complements of countable subsets of $X$ are open). Then
\begin{enumerate}[\rm (a)]
    \item $\mathcal C(X_{coc})=\{\emptyset, X\}\bigcup X^{(\leqslant\omega)}$, $X_{coc}$ is $T_1$ and hence a $d$-space, and the specialization order on $X_{coc}$ is the discrete order.
   \item Neither $X_{coc}$ nor $P_S(X_{coc})$ is first-countable.

For a point $x\in X$, suppose that there is a countable base $\{X\setminus C_n : n\in \mathbb{N}, C_n\in X^{(\leqslant\omega)}\}$ at $x$ in $X_{coc}$. Let $C=\bigcup_{n\in \mathbb{N}}C_n$. Then $C\in X^{(\leqslant\omega)}$. Select $t\in X\setminus (C\cup\{x\})$ and let $U=X\setminus \{t\}$. Then $x\in U
\in \mathcal O(X_{coc})$. But $X\setminus C_n\nsubseteq U$  for every $n\in \mathbb{N}$, a contradiction. Thus $X_{coc}$ is not first-countable. Since the first-countability is a hereditary property and $X_{coc}$ is homeomorphic to the subspace $\mathcal S^u(X_{coc})$ of $P_S(X_{coc})$ (see Remark \ref{xi embdding} or Proposition \ref{soberification first-countable implies X is also} below), $P_S(X_{coc})$ is not first-countable.

 \item $\ir_c(X_{coc})=\{\overline{\{x\}} : x\in X\}\cup\{X\}=\{\{x\} : x\in X\}\cup\{X\}$. Therefore, $X_{coc}$ is not sober.

\item $\mk (X_{coc})=X^{(<\omega)}\setminus \{\emptyset\}$ and $\ii~\!K=\emptyset$ for all $K\in \mk (X_{coc})$, and hence $X_{coc}$ is not locally compact.

    Clearly, every finite subset is compact. Conversely, if $C\subseteq X$ is infinite, then $C$ has an infinite countable subset $\{c_n : n\in\mn\}$. Let $C_0=\{c_n : n\in\mn\}$ and $U_m=(X\setminus C_0)\cup \{c_m\}$ for each $m\in \mn$. Then $\{U_n : n\in\mn\}$ is an open cover of $C$, but has no finite subcover. Hence $C$ is not compact. Thus $\mk (X_{coc})=X^{(<\omega)}\setminus \{\emptyset\}$. Clearly, $\ii~\!K=\emptyset$ for all $K\in \mk (X_{coc})$. Hence $X_{coc}$ is not locally compact.

    \item $X_{coc}$ is well-filtered and not core-compact.

    Suppose that $\{F_d : d\in D\}\subseteq \mk (X_{coc})$ is a filtered family and $U\in \mathcal O(X_{coc})$ with $\bigcap_{d\in D}F_d\subseteq U$. As $\{F_d : d\in D\}$ is filtered and all $F_d$ are finite, $\{F_d : d\in D\}$ has a least element $F_{d_0}$, and hence $F_{d_0}=\bigcap_{d\in D}F_d\subseteq U$, proving that $X_{coc}$ is well-filtered. By (d) and Theorem \ref{SoberLC=CoreCNew}, $X_{coc}$ is not core-compact.

\item $\mk (X_{coc})$ is a Noetherian dcpo and hence $\Sigma~\!\!\mk (X_{coc})=(\mk (X_{coc}), \alpha(\mk (X_{coc}))$ is first-countable.

Clearly,  $\mk (X_{coc})=X^{(<\omega)}\setminus \{\emptyset\}$ (with the Smyth order) is a Noetherian dcpo and $\sigma(\mk (X_{coc}))=\alpha(\mk (X_{coc}))$. For any $F\in \mk (X_{coc})=X^{(<\omega)}\setminus \{\emptyset\}$, $\{\ua_{\mk (X_{coc})}F\}$ is a base at $F$ in $\Sigma~\!\! \mk(X_{coc})$. Hence $\Sigma~\!\! \mk(X_{coc})$ is first-countable.

\item The upper Vietoris topology and the Scott topology on $\mk (X_{coc})$ do not agree.

By (e) and Corollary \ref{wf space Vietoris less Scott}, $\mathcal O(P_S(X_{coc}))\subseteq \sigma (\mk (X_{coc})$. For $F\in X^{(<\omega)}\setminus \{\emptyset\}$, $\ua_{\mk (X_{coc})}F\in \alpha(\mk (X_{coc}))=\sigma(\mk (X_{coc}))$ but $\ua_{\mk (X_{coc})}F\not\in \mathcal O(P_S(X_{coc}))$ since there is no $G\in X^{(<\omega)}$ with $F\in \Box (X\setminus G)=(X\setminus G)^{(<\omega)}\setminus \{\emptyset\}\subseteq \ua_{\mk (X_{coc})}F$. Thus $\sigma(\mk (X_{coc}))\nsubseteq\mathcal O(P_S(X_{coc}))$.

\item The Scott power space $\Sigma~\!\! \mk(X_{coc})$ is a sober $c$-space. So it is Rudin and well-filtered determined.

$\mk (X_{coc})=X^{(<\omega)}\setminus \{\emptyset\}$ (with the Smyth order) is a Noetherian dcpo and hence it is an algebraic domain. By Theorem \ref{algebraic is continuous} and Proposition \ref{quasicontinuous domain is sober}, $\Sigma~\!\! \mk(X)$ is a sober $c$-space. Hence by Theorem \ref{soberequiv}, $\Sigma~\!\! \mk(X)$ is Rudin and well-filtered determined.

\item  $X_{coc}$ is neither a Rudin space nor a $\mathsf{WD}$ space.

    By (c)(e) and Theorem \ref{soberequiv}, $X_{coc}$ is neither a Rudin space nor a $\mathsf{WD}$ space.

\item  The Smyth power space $P_S(X_{coc})$ is well-filtered but non-sober. Hence it is neither a Rudin space nor a $\mathsf{WD}$ space.

By (c)(e), Theorem \ref{Schalk-Heckman-Keimel theorem} and Theorem \ref{Smythwf}, $P_S(X_{coc})$ is well-filtered and non-sober. Hence $P_S(X_{coc})$ is neither a Rudin space nor a $\mathsf{WD}$ space by Theorem \ref{soberequiv}.

\item  $P_S(X_{coc})$ is not core-compact.

By (j) and Theorem \ref{soberequiv}, it needs only to show that $P_S(X_{coc})$ is not locally compact. Assume, on the contrary, that $P_S(X_{coc})$ is not locally compact. For $x\in X$ and $U\in \mathcal O(X_{coc})$ with $x\in U$, then by the local compactness of $P_S(X_{coc})$, there is $V\in \mathcal O(X_{coc})$ and $\mathcal K\in \mk (X_{coc})$ such that $\ua x\in \Box V\subseteq \mathcal K\subseteq \Box U$. Let $K=\bigcup \mathcal K$. Then by Lemma \ref{K union}, $K\in \mk (X)$ and $x\in V=\bigcup \Box V\subseteq K\subseteq \bigcup \Box U=U$. It follows that $X_{coc}$ is locally compact, which is in contradiction with (e). Thus $P_S(X_{coc})$ is not core-compact.
\end{enumerate}
\end{example}

The following example shows there is even a second-countable Noetherian $T_0$ space $X$ such that the Scott power space $\Sigma ~\!\!\mk (X)$ is a second-countable sober space but $X$ is not well-filtered (and hence not sober).

\begin{example}\label{Scott sober not implies X is wf}
	Let $P=\mathbb{N}\cup\{\infty\}$ and define an order on $P$ by $x\leq_P y$ iff $x=y$ or $x\in\mathbb{N}$ and $y=\infty$ (see Figure 4).

 \begin{figure}[ht]
	\centering
	\includegraphics[height=3cm,width=4cm]{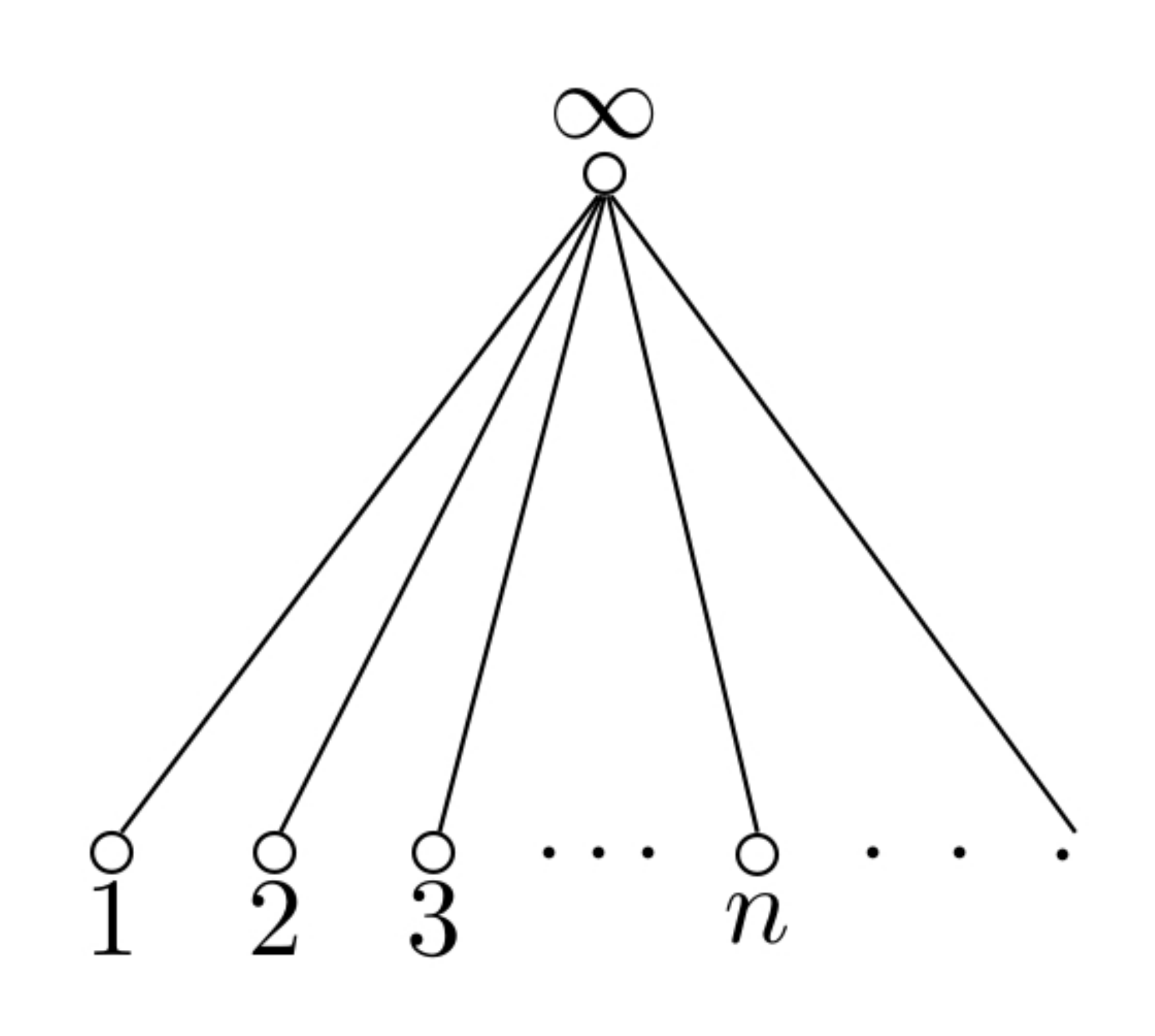}
	\caption{The poset P}
\end{figure}

\noindent Let $\tau=\{(\mathbb{N}\setminus F)\cup \{\infty\} : F\in \mathbb{N}^{(<\omega)}\}\cup\{\emptyset, P\}\cup \{\{\infty\}\}$. It is straightforward to verify that $\tau$ is a $T_0$ topology on $P$ and the specialization order of $(P, \tau)$ agrees with the original order on $P$. Now we have
\begin{enumerate}[\rm (a)]
    \item $\mathcal C((P, \tau))=\mathbb{N}^{(<\omega)}\cup \{\emptyset, P\}\cup \{\mathbb{N}\}$.
    \item $\ir_c((P, \tau))=\{\overline{\{n\}}=\{n\} : n\in \mathbb{N}\}\cup\{\overline{\{\infty\}}=P\}\cup\{\mathbb{N}\}$ and hence $(P, \tau)$ is not sober.
    \item $\mk ((P, \tau))=\{A\cup \{\infty\} : A\subseteq \mathbb{N}\}$.
    \item $(P, \tau)$ is not well-filtered.

Let $\mathcal K=\{(\mathbb{N}\setminus F)\cup\{\infty\} : F\in \mathbb{N}^{(<\omega)}\}$. Then $\mathcal K\subseteq \mk ((P, \tau))$ is a filtered family and $\bigcap \mathcal K=\{\infty\}\in \tau$. But there is no $F\in \mathbb{N}^{(<\omega)}$ with $(\mathbb{N}\setminus F)\cup\{\infty\}=\{\infty\}$. Thus  $(P, \tau)$ is not well-filtered. In fact, $(P, \tau)$ is not weak well-filtered in the sense of \cite{LL-2017}.

    \item $(P, \tau)$ is Noetherian and second-countable and hence it is a Rudin space.

    Since $|\tau|=\omega$, $(P, \tau)$ is second-countable. As every subset of $P$ is compact in $(P, \tau)$, the space $(P, \tau)$ is a Noetherian space (and hence a locally compact space). Hence by Proposition \ref{LCrudin and core-compact is WD} $(P, \tau)$ is a Rudin space.

    \item $\Sigma~\!\!\mk ((P, \tau))$ is a second-countable sober space.

    Clearly, $\mk ((P, \tau))$ is isomorphic with the algebraic lattice $2^{\mathbb{N}}$ (with the order of set inclusion) via the poset isomorphism $\varphi : \mk ((P, \tau))\rightarrow 2^{\mathbb{N}}$ defined by $\varphi (A\cup \{\infty\})=\mathbb{N}\setminus A$ for each $A\in 2^{\mathbb{N}}$ (note that the order on $\mk ((P, \tau))$ is the Smyth order). Hence $\Sigma~\!\!\mk ((P, \tau))\cong \Sigma~\!\!2^{\mathbb{N}}$. Clearly, $2^{\mathbb{N}}$ is an algebraic lattice, whence by Theorem \ref{algebraic is continuous} and Proposition \ref{quasicontinuous domain is sober}, $\Sigma~\!\!2^{\mathbb{N}}$ is sober and hence $\Sigma~\!\!\mk ((P, \tau))$ is sober. Clearly, $\Sigma~\!\!2^{\mathbb{N}}$ is second-countable since $\{\ua_{2^{\mathbb{N}}} F : F\in (2^{\mathbb{N}})^{(<\omega)}\}$ is a countable base of $\Sigma~\!\!2^{\mathbb{N}}$. So $\Sigma~\!\!\mk ((P, \tau))$ is second-countable.

    \item $P_S((P, \tau))$ is second-countable by Proposition \ref{Smyth CII}.

    \item $\sigma(\mathsf{K}((P, \tau))\subseteq \mathcal O(P_S((P, \tau)))$ but $\mathcal O(P_S((P, \tau))\nsubseteq \sigma(\mathsf{K}((P, \tau)))$.

    Since $(P, \tau)$ is locally compact, $\sigma(\mathsf{K}((P, \tau))\subseteq \mathcal O(P_S((P, \tau)))$ by Lemma \ref{local comp imply V topol finer than Scott topol}. Clearly, $\Box \{\infty\}=\{\{\infty\}\}\in \mathcal O(P_S((P, \tau))$. Now we show that $\Box \{\infty\}\not\in \sigma(\mathsf{K}((P, \tau)))$. By Lemma \ref{Kmeet}, $\bigvee \{F\cup\{\infty\} : F\in (\mathbb{N})^{(<\omega)}\setminus \{\emptyset\}\}=\bigcap \{F\cup\{\infty\} : F\in (\mathbb{N})^{(<\omega)}\setminus \{\emptyset\}\}=\{\infty\}\in \Box \{\infty\}$, but there is no $F\in (\mathbb{N})^{(<\omega)}\setminus \{\emptyset\}$ with
    $F\cup\{\infty\}\in \Box \{\infty\}=\{\{\infty\}\}$. Thus $\Box \{\infty\}\not\in \sigma(\mathsf{K}((P, \tau)))$.
\end{enumerate}
\end{example}

In the following we will construct a sober space $X$ for which its Scott power space is non-sober (see Theorem \ref{Scott power space of a sober space is non-sober} below).

Let $\mathcal{L}=\mathbb{N}\times \mathbb{N}\times (\mathbb{N}\cup \{\infty\})$, where $\mathbb{N}$ is the set of natural numbers with the usual order. Define an order $\leq$ on $\mathcal L$ as follows:

$(i_1, j_1, k_1)\leq (i_2, j_2, k_2$ if and only if:

(1) $i_1=i_2, j_1=j_2, k_1\leq k_2\leq \infty$; or

(2) $i_2=i_1+1, k_1\leq j_2, k_2=\infty$.

 $\mathcal L$ is a known dcpo constructed by Jia in \cite[Example 2.6.1]{jia-2018}. It can be easily depicted as in Figure 5 taken from \cite{jia-2018}.

 \begin{figure}[ht]
	\centering
	\includegraphics[height=4cm,width=7cm]{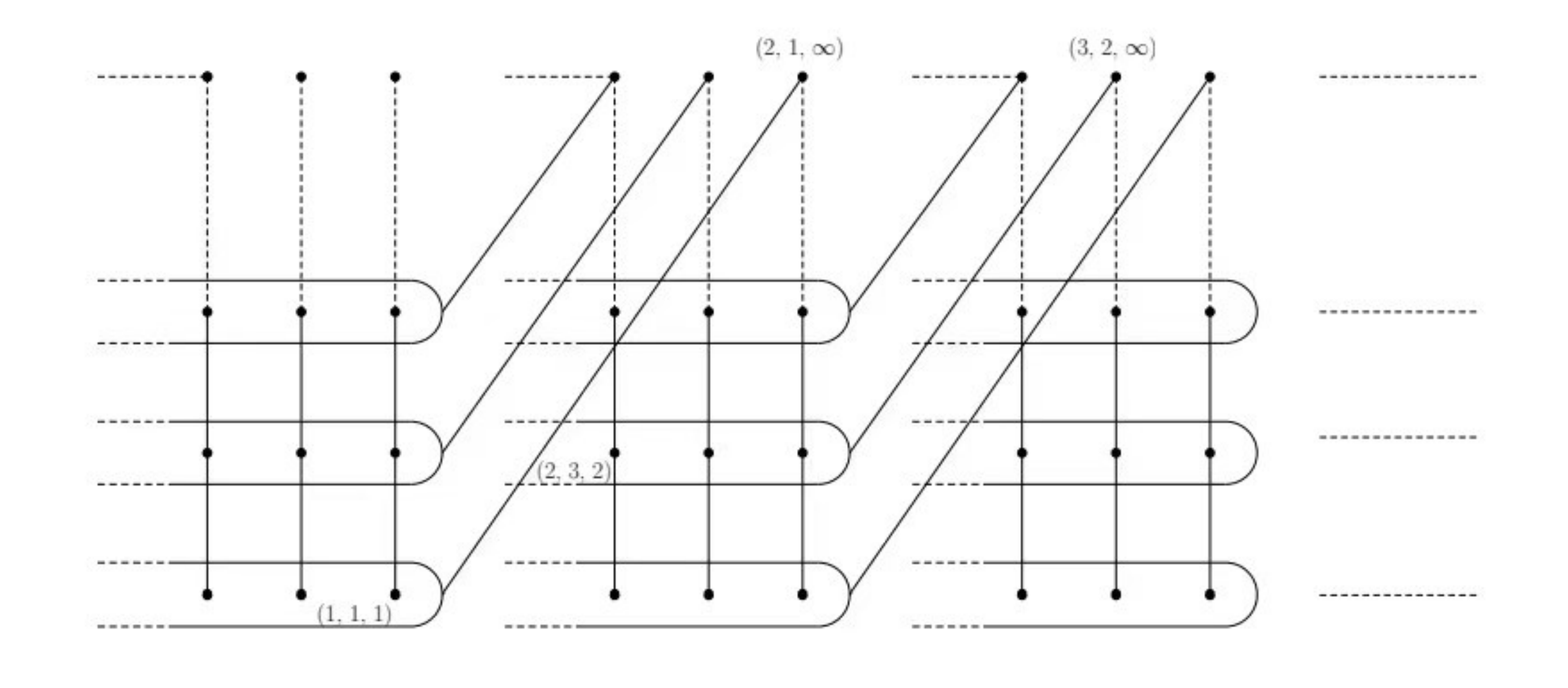}
	\caption{A non-sober well-filtered dcpo $\mathcal L$}
\end{figure}

For each $(n, m)\in \mathbb{N}\times \mathbb{N}$, let

$\mathcal L_n=\{(n, j, l) : j\in \mathbb{N}, l\in\mathbb{N}\cup \{\infty\}\}$,

$\mathcal L_n^{\infty}=\{(n, j, \infty) : j\in \mathbb{N}\}$,

$\mathcal L^{\infty}=\bigcup_{n\in \mathbb{N}}\mathcal L_n^{\infty}=\{(i, j, \infty) : (i, j)\in \mathbb{N}\times\mathbb{N}\}$ (the set of all maximal elements of $\mathcal L$),

$\mathcal L^{<\infty}=\mathcal L\setminus \mathcal L^{\infty}$ (the set of all elements of finite height),

$\mathcal L_{(\leq n)}=\bigcup\limits_{i=1}^{n}\mathcal L_i=\{(i, j, l) : i\leq n, j\in \mathbb{N}, l\in \mathbb{N}\cup\{\infty\}\}$,

$\mathcal L_{(< n+1)}^{\infty}=\mathcal L_{(\leq n)}^{\infty}=\bigcup\limits_{i=1}^{n}\mathcal L_i^{\infty}=\{(i, j, \infty) : i\leq n, j\in \mathbb{N}\}$, and

$\mathcal L_{(n, \geq m)}^{\infty}=\{(n, j, \infty) : j\geq m\}$.

\begin{lemma}\label{key lemma about Jia space} Suppose that $D$ is an infinite directed subset of $\mathcal L$. Then there is a unique $(i, j, \infty)\in \mathcal L^{\infty}$ such that $(i, j, \infty)$ is a largest element of $D$ or the following two conditions are satisfied:
\begin{enumerate}[\rm (i)]
\item $(i, j, \infty)\not\in D$, and
\item for each $d=(i_d, j_d, l_d)\in D$, $i_d=i, j_d=j$ and $l_d<\infty$ (i.e., $l_d\in\mathbb{N}$).
\end{enumerate}
\end{lemma}
\begin{proof} If there is $d_0=(i, j, \infty)\in D\cap \mathcal L^{\infty}$, then for each $d=(i_d, j_d, l_d)\in D$, there is $d^*=(i_{d^*}, j_{d^*}, l_{d^*})\in D$ such that $d_0=(i, j, \infty)\leq d^*=(i_{d^*}, j_{d^*}, l_{d^*})$ and $d=(i_d, j_d, l_d)\leq d^*=(i_{d^*}, j_{d^*}, l_{d^*})$. Hence $l_{d^*}=\infty, i_{d^*}=i, j_{d^*}=j$ (i.e., $d^*=d_0$) and $d\leq d^*=d_0$. Hence $d_0=(i, j, \infty)$ is the (unique) largest element of $D$.

Now suppose that $D\cap \mathcal L^{\infty}=\emptyset$, that is, $D\subseteq \mathbb{N}\times \mathbb{N}\times \mathbb{N}$. Select a $d_1=(i_{d_1}, j_{d_1}, l_{d_1})\in D$. Then for each $d=(i_d, j_d, l_d)\in D$, by the directedness of $D$, there is $d^{\prime}=(i_{d^{\prime}}, j_{d^{\prime}}, l_{d^{\prime}})\in D$ such that $d_1=(i_{d_1}, j_{d_1}, l_{d_1})\leq d^{\prime}=(i_{d^{\prime}}, j_{d^{\prime}}, l_{d^{\prime}})$ and $d=(i_d, j_d, l_d)\leq d^{\prime}=(i_{d^{\prime}}, j_{d^{\prime}}, l_{d^{\prime}})$. Hence $i_{d^{\prime}}=i_d=i_{d_0}, j_{d^{\prime}}=j_d=j_{d_0}$ and $l_{d_0}\leq l_{d^{\prime}}, l_d\leq l_{d^{\prime}}$. Let $i=i_{d_1}$ and $j=j_{d_1}$. Then $D\subseteq \{(i, j, l) : l\in \mathbb{N}\}$. Clearly, $(i, j, \infty)$ is the unique element of $\mathcal L^{\infty}$ satisfying conditions (i) and (ii), and $(i, j, \infty)=\bigvee_{\mathcal L} D$.
\end{proof}

For any infinite directed subset $D$ of $\mathcal L$, by Lemma \ref{key lemma about Jia space}, $D$ is contained in $\da (i, j, \infty)$ for some $i, j\in \mathbb{N}$ with its supremun being $(i, j, \infty)$. Hence $\mathcal L$ is a dcpo.

\begin{corollary}\label{key corollary about Jia space} Let $A$ be a nonempty subset of $\mathcal L$ for which $A=\da \mathrm{max}(A)$. For an infinite directed subset $D$ of $A$, if $D$ has no largest element, then there is a unique $(i, j, \infty)\in \mathcal L^{\infty}$ such that
\begin{enumerate}[\rm (1)]
\item $D\subseteq \{(i, j, l) : l\in \mathbb{N}\}$,
\item $(i, j, \infty)=\bigvee_{\mathcal L} D$, and
\item $\mathrm{max}(A)\cap \mathcal L_{i+1}^{\infty}$ is infinite.
\end{enumerate}
\end{corollary}
\begin{proof} Since $D$ has no largest element, by Lemma \ref{key lemma about Jia space}, there is a unique $(i, j, \infty)\in \mathcal L^{\infty}$  such that the following two conditions are satisfied:
\begin{enumerate}[\rm (i)]
\item $(i, j, \infty)\not\in D$, and
\item for each $d=(i_d, j_d, l_d)\in D$, $i_d=i, j_d=j$ and $l_d<\infty$ (i.e., $l_d\in\mathbb{N}$).
\end{enumerate}
\noindent So $D\subseteq \{(i, j, l) : l\in \mathbb{N}$ and $(i, j, \infty)=\bigvee_{\mathcal L} D$. Now we show that $\mathrm{max}(A)\cap \mathcal L_{i+1}^{\infty}$ is infinite. For each $d=(i, j, l_d)\in D$, by $A=\da \mathrm{max}(A)$, there is $(i(d), j(d), l(d))\in \mathrm{max}(A)$ with $d=(i, j, l_d)\leq (i(d), j(d), l(d))$. If $i(d)=i$, then $j(d)=j$ and $l_d\leq l(d)$. Since $D\subseteq A$, $\{l_{d^{\prime}} : d^{\prime}=(i, j, l_{d^{\prime}})\in D\}\subseteq \mathbb{N}$ is infinite and $(i, j, l(d))=(i(d), j(d), l(d))\in \mathrm{max}(A)$, we have that $l(d)=\infty$, which is in contradiction with condition (i). Therefore, $i(d)=i+1$ and hence $l(d)=\infty$ and $l_d\leq j(d)$ by $(i, j, l_d)\leq (i(d), j(d), l(d))=(i+1, j(d), l(d))$. Since $\{l_d : d=(i_d, j_d, l_d)=(i, j, l_d)\in D\}\subseteq \mathbb{N}$ is infinite, $\{(i(d), j(d), l(d))=(i+1, j(d), \infty) : d\in D\}\subseteq \mathrm{max}(A)\cap \mathcal L^{\infty}$ is infinite (note that $l_d\leq j(d)$ for each $d\in D$). Thus $\mathrm{max}(A)\cap \mathcal L_{i+1}^{\infty}$ is infinite.
\end{proof}

\begin{lemma}\label{closed set in Jia space} Let $A\subseteq \mathcal L$ be a nonempty set and $A\neq \mathcal L$. Then $A$ is Scott closed if and only if $A=\da \mathrm{max}(A)$ and one of the following three conditions are satisfied:
\begin{enumerate}[\rm (1)]
\item  $A\subseteq \mathcal L^{<\infty}$ (or equivalently, $\mathrm{max}(A)\subseteq \mathcal L^{<\infty}$).
\item $A\cap \mathcal L^{\infty}\neq \emptyset$ and $|A\cap L_i^{\infty}|<\omega$ for each $i\in \mathbb{N}$.
\item $i(A)=\mathrm{max}\{i\in \mathbb{N} : |A\cap L_i^{\infty}|=\omega \}$ exists and $\mathcal L_{n}\subseteq A$ for each $n\leq i(A)-1$.
\end{enumerate}
\end{lemma}

\begin{proof} Suppose that $A$ is Scott closed. Then $A=\da \mathrm{max} (A)$ by Lemma \ref{d-space max point}. Now we show that $A$ satisfies one of conditions (1)-(3). If neither condition (1) nor condition (2) holds. Then there is some $i_0\in \mathbb{N}$ such that $A\cap L_{i_0}^{\infty}$ is infinite. We first show that $A$ satisfies the following property Q:

\textbf{(Q)} For $n\in \mathbb{N}$, if $A\cap L_{n}^{\infty}$ is infinite, then $\mathcal L_{i}\subseteq A$ for each $i\leq n-1$.

If $n=1$, then $\mathcal L_{n-1}=\mathcal L_0=\emptyset\subseteq A$. Now we assume $2\leq n$. For each $(j, l)\in \mathbb{N}\times \mathbb{N}$, since $A\cap L_{n}^{\infty}$ is infinite (i.e., $\{j^{\prime}\in \mathbb{N} : (n, j^{\prime}, \infty)\in A\}$ is infinite), there is $j^{\prime}\in \mathbb{N}$ such that $(n, j^{\prime}, \infty)\in A$ and $l\leq j^{\prime}$, whence $(n-1, j, l)\leq (n, j^{\prime}, \infty)$. Thus $\{(n-1, j, l) : (j, l)\in \mathbb{N}\times \mathbb{N}\}\subseteq \da (A\cap L_{n}^{\infty})\subseteq \da A=A$. For each $j\in \mathbb{N}$, since $(n-1, j, \infty)=\bigvee_{l\in \mathbb{N}}(n-1, j, l)$ and $A$ is Scott closed, we have $(n-1, j, \infty)\in A$. Hence $\mathcal L_{n-1}\subseteq A$. In particular, $\mathcal L_{n-1}^{\infty}\subseteq A$. Then by induction we get that $\mathcal L_i\subseteq A$ for any $1\leq i\leq n-1$.

By property Q, if $\{i\in \mathbb{N} : |A\cap L_i^{\infty}|=\omega \}$ is infinite, then for each $n\in \mathbb{N}$, $\mathcal L_{n-1}\subseteq A$. Hence $\mathcal L=\bigcup_{n\in \mathbb{N}}\mathcal L_n\subseteq A$, which contradicts $A\neq \mathcal L$. Hence $\{i\in \mathbb{N} : |A\cap L_i^{\infty}|=\omega \}$ is a nonempty finite subset of $\mathbb{N}$ and hence $i(A)=\mathrm{max}\{i\in \mathbb{N} : |A\cap L_i^{\infty}|=\omega \}$ exists. By property Q we have $\mathcal L_{i}\subseteq A$ for each $i\leq i(A)-1$, proving that condition (3) holds.

Conversely, assume that $A=\da \mathrm{max}(A)$ and one of conditions (1)-(3) is satisfied. We will show that $A$ is Scott closed.

\textbf{Case 1.} $A\subseteq \mathcal L^{<\infty}$ (resp., $A\cap \mathcal L^{\infty}\neq \emptyset$ and $|A\cap L_i^{\infty}|<\omega$ for each $i\in \mathbb{N}$).

Suppose that $D$ is a directed subset of $A$. If $D$ has no largest element, then $D$ is infinite, whence by Corollary \ref{key corollary about Jia space}, there is a unique $(i, j, \infty)\in \mathcal L^{\infty}$ such that $D\subseteq \{(i, j, l) : l\in \mathbb{N}\}$, $(i, j, \infty)=\bigvee_{\mathcal L} D$ and $\mathrm{max}(A)\cap \mathcal L_{i+1}^{\infty}$ is infinite, being a contradiction with $A\subseteq \mathcal L^{<\infty}$ (resp., $A\cap \mathcal L^{\infty}\neq \emptyset$ and $|A\cap \mathcal L_i^{\infty}|<\omega$ for each $i\in \mathbb{N}$).  Therefore, $D$ has a largest element $d_0$ and hence $\vee D=d_0\in A$. Thus $A$ is Scott closed.

\textbf{Case 2.} $i(A)=\mathrm{max}\{i\in \mathbb{N} : |A\cap L_i^{\infty}|=\omega \}$ exists and $\mathcal L_{n}\subseteq A$ for each $n\leq i(A)-1$.

 For an infinite directed subset $D$ of $A$, if $D$ has a largest element, then clearly $\bigvee_{\mathcal L} D\in A$. If $D$ has no largest element, then by Corollary \ref{key corollary about Jia space}, there is a unique $(i, j, \infty)\in \mathcal L^{\infty}$ such that $D\subseteq \{(i, j, l) : l\in \mathbb{N}\}$, $(i, j, \infty)=\bigvee_{\mathcal L} D$ and $\mathrm{max}(A)\cap \mathcal L_{i+1}^{\infty}$ is infinite. Since $i(A)=\mathrm{max}\{i\in \mathbb{N} : |A\cap L_i^{\infty}|=\omega \}$ exists and $\mathcal L_{n}\subseteq A$ for each $n\leq i(A)-1$, we have $i+1\leq i(A)$ and hence $\bigvee_{\mathcal L}D=(i, j, \infty)\in \mathcal L_{i}\subseteq A$. So $A$ is Scott closed.
 \end{proof}

\begin{lemma}\label{Irreducible set in Jia space} $\ir_c(\Sigma~\!\!\mathcal L)=\{\overline{\{x\}}=\da x : x\in \mathcal L\}\cup\{\mathcal L\}$.
\end{lemma}
\begin{proof} Clearly, $\{\overline{\{x\}}=\da x : x\in \mathcal L\}\subseteq \ir_c(\Sigma~\!\!\mathcal L)$. It was proved in \cite[Example 2.6.1]{jia-2018} that $\mathcal L\in \ir_c(\Sigma~\!\!\mathcal L)$. Suppose that $A\in \ir_c(\Sigma~\!\!\mathcal L)$ and $A\neq \mathcal L$. Then by Lemma \ref{closed set in Jia space}, $A=\da \mathrm{max}(A)=\da (\mathrm{max}(A)\cap \mathcal L^{\infty})\cup \da (\mathrm{max}(A)\cap \mathcal L^{<\infty})$ ($\mathrm{max}(A)\cap \mathcal L^{\infty}$ or $\mathrm{max}(A)\cap \mathcal L^{<\infty}$ may be the empty set) and one of the following three conditions are satisfied:
\begin{enumerate}[\rm (i)]
\item  $A\subseteq \mathcal L^{<\infty}$ (or equivalently, $\mathrm{max}(A)\subseteq \mathcal L^{<\infty}$).
\item $A\cap \mathcal L^{\infty}\neq \emptyset$ and $|A\cap L_i^{\infty}|<\omega$ for each $i\in \mathbb{N}$.
\item $i(A)=\mathrm{max}\{i\in \mathbb{N} : |A\cap L_i^{\infty}|=\omega \}$ exists and $\mathcal L_{n}\subseteq A$ for each $n\leq i(A)-1$.
\end{enumerate}

Let $B=\da (\mathrm{max}(A)\cap \mathcal L^{\infty})$ and $C=\da (\mathrm{max}(A)\cap \mathcal L^{<\infty})$. Then $B=\da \mathrm{max}(B)$ and $C=\da \mathrm{max}(C)$. Clearly, $C\subseteq \mathcal L^{<\infty}$. Hence $C$ is Scott closed by Lemma \ref{closed set in Jia space}. If $B\neq\emptyset$ (i.e., $\mathrm{max}(A)\cap \mathcal L^{\infty}\neq\emptyset$), then as $A$ satisfies one of conditions (2) and (3) of Lemma \ref{closed set in Jia space}, $B$ also satisfies one of conditions (2) and (3) of Lemma \ref{closed set in Jia space}, and hence by Lemma \ref{closed set in Jia space} again, $B$ is Scott closed. By the irreducibility of $A$, we have $A=B$ or $A=C$.

\textbf{Case 1.} $\mathrm{max}(A)\cap \mathcal L^{\infty}\neq \emptyset$ and $A=B=\da (\mathrm{max}(A)\cap \mathcal L^{\infty})$.

If $i(A)=\mathrm{max}\{i\in \mathbb{N} : |A\cap L_i^{\infty}|=\omega \}$ exists and $\mathcal L_{n}\subseteq A$ for each $n\leq i(A)-1$, then there is $(i(A), j, \infty)\in A\cap \mathcal L_{i(A)}^{\infty}$. By Lemma \ref{key lemma about Jia space} we can easily verify that  $\bigcup\limits_{n=1}^{i(A)-1}\da (A\cap \mathcal L_i^{\infty}\setminus \{(i(A), j, \infty)\})$ is Scott closed. Clearly, $A=(\bigcup\limits_{n=1}^{i(A)-1} \da (A\cap L_i^{\infty}\setminus \{(i(A), j, \infty)\}))\cup \da (i(A), j, \infty)$. By the irreducibility of $A$ and $A\neq\da (i(A), j, \infty)$ (since $A\cap \mathcal L_{i(A)}^{\infty}$ is infinite), we have $A=\bigcup\limits_{n=1}^{i(A)-1}\da (A\cap L_i^{\infty}\setminus \{(i(A), j, \infty)\})$.  which contradicts $(i(A), j, \infty)\not\in \bigcup\limits_{n=1}^{i(A)-1} \da (A\cap L_i^{\infty}\setminus \{(i(A), j, \infty)\})$.

So $|A\cap L_i^{\infty}|<\omega$ for each $i\in \mathbb{N}$. Choose a point $(i, j, \infty)\in \mathrm{max}(A)\cap \mathcal L^{\infty}$. By Lemma \ref{key lemma about Jia space} we can easily check that $\da (\mathrm{max}(A)\cap \mathcal L^{\infty}\setminus \{(i, j, \infty)\})$ is Scott closet and $A=\da (\mathrm{max}(A)\cap \mathcal L^{\infty}\setminus \{(i, j, \infty)\})\cup \da (i, j, \infty)$. By the irreducibility of $A$, we have $A=\da (i, j, \infty)$ or $A=\da (\mathrm{max}(A)\cap \mathcal L^{\infty}\setminus \{(i, j, \infty)\})$ (which contradicts $(i, j, \infty)\not\in \da (\mathrm{max}(A)\cap \mathcal L^{\infty}\setminus \{(i, j, \infty)\})$). Thus $A=\da (i, j, \infty)$ and $\mathrm{max}(A)\cap \mathcal L^{\infty}=\{(i, j, \infty)\}$.

\textbf{Case 2.} $A=C=\da (\mathrm{max}(A)\cap \mathcal L^{<\infty})$.

Choose a point $(i, j, l)\in \mathrm{max}(A)$. Then by Lemma \ref{closed set in Jia space}, $\da (\mathrm{max}(A)\setminus \{(i, j, l)\})$ is Scott closed and $A=\da (\mathrm{max}(A)\setminus \{(i, j, l)\})\cup \da (i, j, l)$. It follows that $A=\da (i, j, l)$ or $A=\da (\mathrm{max}(A)\setminus \{(i, j, l)\})$, which contradicts $(i, j, l)\not\in \da (\mathrm{max}(A)\setminus \{(i, j, l)\})$. So $A=\da (i, j, l)$ and $\mathrm{max}(A)=\{(i, j, l)\}$.

It follows from the above that $\ir_c(\Sigma~\!\!\mathcal L)=\{\overline{\{x\}}=\da x : x\in \mathcal L\}\cup\{\mathcal L\}$.

\end{proof}

\begin{lemma}\label{some propertie of Jia space} (\cite[Example 2.6.1]{jia-2018}) For a nonempty saturated subset $K\subseteq \mathcal L$, $K$ is compact in $\Sigma~\!\!\mathcal L$ if and only if the following three conditions are satisfied:
\begin{enumerate}[\rm (1)]
\item $\mathrm{min}(K)\cap \mathcal L^{<\infty}$ is finite,
\item there exists $i_0, j_0\in \mathbb{N}$ such that $(\mathcal L_{(<i_0)}^{\infty}\cup \mathcal L_{(i_0, \geq j_0)}^{\infty})\cap K=\{(i_0, j_0, \infty)\}$, and
\item the set $\{n\in \mathbb{N} : \mathcal L_n^{\infty}\cap K\neq\emptyset\}$ is finite.
\end{enumerate}
\end{lemma}

\begin{corollary}\label{filtered family meet not empty}  For any filtered family $\{K_d : d\in D\}\subseteq \mk (\Sigma~\!\!\mathcal L)$, $\bigcap_{d\in D}K_d\neq\emptyset$.
\end{corollary}

\begin{proof} We can assume that $D$ is directed and $K_{d_2}\subseteq K_{d_1}$ iff $d_1\leq d_2$ (indeed, $D$ can be defined an order by $d_1\leq d_2$ iff $K_{d_2}\subseteq K_{d_1}$). For each $d\in D$, by Lemma \ref{some propertie of Jia space}, there exist $i_d, j_d\in \mathbb{N}$ such that $(\mathcal L_{(<i_d)}^{\infty}\cup \mathcal L_{(i_d, \geq j_d)}^{\infty})\cap K_d=\{(i_d, j_d, \infty)\}$ and the set $\mathbb{N}_d=\{n\in \mathbb{N} : \mathcal L_n^{\infty}\cap K_d\neq\emptyset\}$ is finite. Select a $d_0\in D$. Then for each $d\in D$ with $d_0\leq d$ (whence $K_d\subseteq K_{d_0}$), we have
that $i_d\in \mathbb{N}_d\subseteq \mathbb{N}_{d_0}$ and $i_{d_0}\leq i_d$ (otherwise, $i_{d_0}> i_d$ would imply that $(i_d, j_d, \infty)\in (\mathcal L_{(<i_d)}^{\infty}\cup \mathcal L_{(i_d, \geq j_d)}^{\infty})\cap K_d\subseteq L_{(<i_{d_0})}^{\infty}\cap K_{d_0}$, which contradicts $L_{(<i_{d_0})}^{\infty}\cap K_{d_0}=\emptyset$). Let $D_{d_0}=\{d\in D : d_0\leq d\}$ and $D_i=\{d\in D_0 : i_d=i\}$ for each $i\in \mathbb{N}_{d_0}$. Since $\mathbb{N}_{d_0}$ is finite, $D_{d_0}$ is directed and $D_{d_0}=\bigcup_{i\in\mathbb{N}_{d_0}} D_i$, there is $i_0\in \mathbb{N}_{d_0}$ such that $D_{i_0}$ is a cofinal subset of $D_{d_0}$ and hence a cofinal subset of $D$, more precisely, for each $d\in D$, there is $d^*\in D$ such that $d^*\in \ua d_0\cap \ua d$ and $i_{d^*}=i_0$.

Clearly, $D_{i_0}$ is also directed. Select a $d_1\in D_{i_0}$. Then $D_{d_1}=\{d\in D_{i_0} : d_1\leq d\}$ is a directed and cofinal subset of $D_{i_0}$ and hence a directed and cofinal subset of $D$. For each $d\in D_{d_1}$ (note that $K_d\subseteq K_{d_1}$), we have
that $i_{d_1}=i_d=i_0$, $(i_d=i_0, j_d, \infty)\in K_d\subseteq K_{d_1}$ and $\mathcal L_{(i_0, \geq j_{d_1})}^{\infty}\cap K_{d_1}=\{(i_0, j_{d_1}, \infty)\}$. It follows that $j_d\leq j_{d_1}$. For each $1\leq j\leq j_{d_1}$, let $\widetilde{D}_{j}=\{d\in D_{d_1} :  j_d=j\}$. Since $\{1, 2, ..., j_{d_1}\}$ is finite, $D_{d_1}$ is directed and $D_{d_1}=\bigcup_{i\in\mathbb{N}_{d_0}} \widetilde{D}_{j}$, there is $1\leq j_0\leq j_{d_1}$ such that $\widetilde{D}_{j_0}$ is a cofinal subset of $D_{d_1}$ and hence a cofinal subset of $D$, more precisely, for each $d\in D$, there is $d^{\prime}\in D$ such that $d^{\prime}\in \ua d_0\cap \ua d_1\cap \ua d$, $i_{d^{\prime}}=i_0$ and $j_{d^{\prime}}=j_0$. It follows that $(i_0, j_0, \infty)\in \bigcap_{d\in D}K_d$.

\end{proof}

\begin{proposition}\label{Jia space WF non-sober} (\cite[Example 2.6.1]{jia-2018})  $\Sigma~\!\!\mathcal L$ is well-filtered but non-sober.
\end{proposition}

 Indeed, $\mathcal L$ is an irreducible closed subset of $\Sigma~\!\!\mathcal L$ but has no largest element, so $\Sigma~\!\!\mathcal L$ is non-sober.

Using Topological Rudin Lemma, Lemma \ref{Irreducible set in Jia space} and Corollary \ref{filtered family meet not empty}, we can give a short proof of the well-filteredness of $\Sigma~\!\!\mathcal L$. Suppose that $\{K_d : d\in D\}\subseteq \mk (\Sigma~\!\!\mathcal L)$ is a filtered family and $U\in \sigma (\mathcal L)$ with $\bigcap_{d\in D}K_d\subseteq U$. Assume, on the contrary, that $K_d\nsubseteq U$ for each $d\in D$ (whence $U\neq \mathcal L$). Then by Lemma \ref{t Rudin}, $\mathcal L\setminus U$ contains a minimal irreducible closed subset $A$ that still meets all
members $K_d$. By Corollary \ref{filtered family meet not empty}, $\bigcap_{d\in D}K_d\neq\emptyset$, whence $U\neq\emptyset$ and $A\neq \mathcal L$. It follows from Lemma \ref{Irreducible set in Jia space} that $A=\overline{\{x\}}$ for some $x\in \mathcal L$. Then $x\in \bigcap_{d\in D}K_d\subseteq U$, which contradicts $x\in A\subseteq \mathcal L \setminus U$. Thus $\Sigma~\!\!\mathcal L$ is well-filtered.

\begin{definition}\label{X plus top 1} Let $X$ be a $T_0$ space for which $X$ is irreducible (i.e., $X\in\ir_c(X)$). Choose a point $\top$ such that $\top\not\in X$. Then $(\mathcal C(X)\setminus \{X\})\cup \{X\cup\{\top\}\}$ (as the set of all closed sets) is a topology on $X\cup\{\top\}$. The resulting space is denoted by $X_{\top}$. Define a mapping $\zeta_{X}: X\rightarrow  X_{\top}$ by  $\zeta_{X}(x)=x$ for each $x\in X$. Clearly, $\eta_{X}$ is a topological embedding.
\end{definition}

As $X$ is $T_0$, $X_{\top}$ is also $T_0$ and $\overline{\{\top\}}=X\cup\{\top\}$ in $X_{\top}$. Hence $\top$ is a largest element of $X_{\top}$ and for $x, y\in X$, $x\leq_X$ iff $x\leq y$ in $X_{\top}$. It is worthy noting that the set $\{\top\}$ is not open in $X_{\top}$.

\begin{remark}\label{X plus top flat} If $X$ is not irreducible, then there exist $A, B\in \mathcal C(X)\setminus \{X\}$ such that $X=A\cup B$, whence $(\mathcal C(X)\setminus \{X\})\cup \{X\cup\{\top\}\}$ is not a topology on $X\cup\{\top\}$.
\end{remark}

\begin{lemma}\label{compact set in space plus a top} Let $X$ be a $T_0$ space for which $X$ is irreducible. Then $\mathsf{K}(X_{\top})=\{G\cup\{\top\} : G\in \mathsf{K}(X)\}\cup\{\{\top\}\}$.
\end{lemma}

\begin{proof} Clearly, $\mathcal O(X_{\top})=\{U\cup\{\top\} : U\in \sigma(\mathcal L)\setminus \{\emptyset\}\}\cup\{\emptyset\}$.

First, if $K\in \mathsf{K}(X_{\top})\setminus \{\{\top\}\}$, then $G=K\setminus \{\top\}$ is a nonempty saturated subset of $X$. Now we verify that $G$ is a compact subset of $X$. Suppose that $\{U_i : i\in I\}\subseteq \mathcal O(X)\setminus \{\emptyset\}$ is an open cover of $G$. Then $\{U_i\cup\{\top\} : i\in I\}\subseteq \mathcal O(X_{\top})$ is an open cover
of $K=G\cup \{\top\}$. By the compactness of $K$, there is $I_0\in I^{(<\omega)}$ such that $K\subseteq \bigcup_{i\in I_0}U_i\cup\{\top\}$, whence $G=K\setminus \{\top\}\subseteq \bigcup_{i\in I_0}U_i$, proving that $G\in \mathsf{K}(\Sigma~\!\!\mathcal L)$.

Conversely, assume that $G\in \mathsf{K}(X)$ and $\{W_j : j\in J\}\subseteq \mathcal O((\Sigma~\!\!\mathcal L)_{\top})\setminus \{\emptyset\}$ is an open cover of $K=G\cup \{\top\}$. Then $K$ is saturated and for each $j\in J$, there is $V_j\in \mathcal O(X)$ such that $W_j=V_j\cup\{\top\}$. Hence $\{V_j : j\in J\}\subseteq \mathcal O(X)$ is an open cover of $G$. By the compactness of $G$, there is $J_0\in J^{(<\omega)}$ such that $G\subseteq \bigcup_{j\in J_0}V_j$, whence $K=G\cup\{\top\}\subseteq \bigcup_{j\in J_0}W_j$. So $K\in \mathsf{K}(X_{\top})$.

Thus $\mathsf{K}(X_{\top})=\{G\cup\{\top\} : G\in \mathsf{K}(X)\}\cup\{\{\top\}\}$.
\end{proof}

\begin{lemma} \label{key lemma}
Suppose that $X$ is a non-sober $T_0$ space for which $\ir_c(X)=\{\overline{\{x\}} : x\in X\}\cup\{X\}$. Then $\langle X_{\top},\zeta_{X} \rangle$ is a sobrification of $X$.
\end{lemma}
\begin{proof} Since $X$ is non-sober and $\ir_c(X)=\{\overline{\{x\}} : x\in X\}\bigcup\{X\}$, $X\neq\overline{\{x\}}$ for every $x\in X$. It is well-known that the space $X^s$ with the canonical mapping $\eta_{X}: X\longrightarrow X^s$, $\eta_{X}(x)=x$, is a sobrification of $X$ (see, for example, \cite[Exercise V-4.9]{redbook}). For $C\in \mathcal C(X)$, we have

$$\Box_{\ir_c(X)} C=\{A\in \ir_c(X) : A\subseteq  C\}=
\begin{cases}
\{\overline{\{c\}} : c\in C\},& C\neq X,\\
\{\overline{\{x\}} : x\in X\}\cup\{X\},& C=X.
	\end{cases}$$

\noindent Define a mapping $f : X^s \rightarrow X_{\top}$ by
$$f (A)=
\begin{cases}
	x & A=\overline{\{x\}}, x\in X,\\
\top & A=X.
	\end{cases}$$

\noindent For each $C\in (\mathcal C(X)\setminus \{X\})\bigcup \{X_{\top}\}$ and $B\in \mathcal C(X)$, we have $f^{-1}(C)=\Box_{\ir_c(X)} C$ and

$$f(\Box_{\ir_c(X)} B)=
\begin{cases}
	B& B\neq X,\\
X\cup\{\top\}& B=X.
	\end{cases}$$
\noindent Thus $f$ is a homeomorphism.

\begin{equation*}
\centerline{
\xymatrix{ X \ar[dr]_{\zeta_{X}} \ar[r]^-{\eta_X}&  X^s\ar@{.>}[d]^{f} & \\
  & X_{\top}  & &
   }}
\end{equation*}

\noindent So $\langle X_{\top},\zeta_{X}=f\circ\eta_X\rangle$ is a sobrification of $X$.
\end{proof}

The following corollary is straightforward from Lemma \ref{Irreducible set in Jia space}, Proposition \ref{Jia space WF non-sober} and Lemma \ref{key lemma}.

\begin{corollary}\label{sobrification of Jia space}  $\langle (\Sigma~\!\!\mathcal L)_{\top},\zeta_{\mathcal L} \rangle$ is a sobrification of $\Sigma~\!\!\mathcal L$, where $\zeta_{\mathcal L}: \Sigma~\!\!\mathcal L\rightarrow  (\Sigma~\!\!\mathcal L)_{\top}$ is defined by  $\zeta_{\mathcal L}(x)=x$ for each $x\in \mathcal L$.
\end{corollary}

Note that although the set $\{\top\}$ is open in $\Sigma~\!\!(\mathcal L)_{\top}$ (or equivalently, $\top$ is a compact element of the dcpo $\mathcal L\cup\{\top\}$), it is not open in $(\Sigma~\!\!\mathcal L)_{\top}$.

By Lemma \ref{Irreducible set in Jia space} and Lemma \ref{compact set in space plus a top}, we get the following.

\begin{corollary}\label{compact sets in Jia space plus a top} $\mathsf{K}((\Sigma~\!\!\mathcal L)_{\top})=\{G\cup\{\top\} : G\in \mathsf{K}(\Sigma~\!\!\mathcal L))\}\cup\{\{\top\}\}$.
\end{corollary}

\begin{lemma}\label{top is open in Jia space plus a top} $\{\top\}$ is a compact element in the dcpo $\mathsf{K}((\Sigma~\!\!\mathcal L)_{\top})$. Hence $\{\top\}$ is open in the Scott space $\Sigma~\!\!\mathsf{K}((\Sigma~\!\!\mathcal L)_{\top})$.
\end{lemma}
\begin{proof} By Proposition \ref{Jia space WF non-sober}, $\Sigma~\!\!\mathcal L$ is well-filtered and hence $\mathsf{K}((\Sigma~\!\!\mathcal L)$ (with the Smyth order) is a dcpo. Hence by Corollary \ref{compact sets in Jia space plus a top} $\mathsf{K}((\Sigma~\!\!\mathcal L)_{\top})$ is a dcpo. Now we show that $\{\top\}\in K(\mathsf{K}((\Sigma~\!\!\mathcal L)_{\top}))$. Suppose that $\{K_d : d\in D\}$ is a directed subset of $\mathsf{K}((\Sigma~\!\!\mathcal L)_{\top})$ and $\{\top\}\sqsubseteq \bigvee_{d\in D}K_d$. Then by Lemma \ref{Kmeet}, $\bigvee_{d\in D}K_d=\bigcap_{d\in D}K_d$ and hence $\bigcap_{d\in D}K_d\subseteq \{\top\}$. It follows that $\bigcap_{d\in D}(K_d\setminus \{\top\})=\emptyset$. By Corollary \ref{filtered family meet not empty} and Corollary \ref{compact sets in Jia space plus a top}, there is $d\in D$ such that $K_d\setminus \{\top\}=\emptyset$, that is, $K_d=\{\top\}$. Thus $\{\top\}\in \sigma(\mathsf{K}((\Sigma~\!\!\mathcal L)_{\top}))$.
\end{proof}

\begin{theorem}\label{Scott power space of a sober space is non-sober} The Scott power space $\Sigma~\!\!\mathsf{K}((\Sigma~\!\!\mathcal L)_{\top})$ of the sober space $(\Sigma~\!\!\mathcal L)_{\top}$ is non-sober.
\end{theorem}

\begin{proof} For simplicity, let $\mathcal A=\{G\cup\{\top\} : G\in \mathsf{K}(\Sigma~\!\!\mathcal L)\}$.

{\bf Claim 1:} $\mathcal A$ is a closed subset of $\Sigma~\!\!\mathsf{K}((\Sigma~\!\!\mathcal L)_{\top})$.

By Corollary \ref{compact sets in Jia space plus a top} and Lemma \ref{top is open in Jia space plus a top}, $\mathcal A$ is Scott closed.

{\bf Claim 2:} $\mathcal A$ is irreducible.

Suppose that $\mathcal U, \mathcal V\in \sigma(\mathsf{K}((\Sigma~\!\!\mathcal L)_{\top}))$ and $\mathcal A\bigcap \mathcal U\neq \emptyset \neq \mathcal A\bigcap \mathcal V$. Then by Corollary \ref{compact sets in Jia space plus a top}, there are some $G_1, G_2\in \mathsf{K}(\Sigma~\!\!\mathcal L))$ such that $G_1\cup\{\top\}\in \mathcal A\bigcap \mathcal U$ and $G_2\cup\{\top\}\in \mathcal A\bigcap \mathcal V$, and hence $(i_1, j_1, \infty)\in G_1$ and $(i_2, j_2, \infty)\in G_2$ for some $(i_1, j_1), (i_2, j_2)\in \mathbb{N}\times \mathbb{N}$. Hence by Corollary \ref{compact sets in Jia space plus a top},$\ua (i_1, j_1, \infty)\cup \{\top\}\in \mathcal A\bigcap \mathcal U$ and $\ua (i_2, j_2, \infty)\cup \{\top\}\in \mathcal A\bigcap \mathcal V$ (note that $\mathcal U, \mathcal V$ are upper sets and $G_1\cup\{\top\}\sqsubseteq \ua (i_1, j_1, \infty)\cup \{\top\}, G_2\cup\{\top\}\sqsubseteq \ua (i_2, j_2, \infty)\cup \{\top\}$). Without loss of generality, we assume $i_1\leq i_2$. Since $\bigvee_{l\in \mathbb{N}}(\ua (i_1, j_1, \infty)\cup \{\top\})=\bigcap_{l\in \mathbb{N}}(\ua (i_1, j_1, \infty)\cup \{\top\})=\ua (i_1, j_1, \infty)\cup \{\top\}\in \mathcal U\in \sigma(\mathsf{K}((\Sigma~\!\!\mathcal L)_{\top}))$, we have some $l_1\in\mathbb{N}$ such that $\ua (i_1, j_1, l_1)\in \mathcal U$. So by Corollary \ref{compact sets in Jia space plus a top}, $\ua (i_1+1, l_1, \infty)\cup\{\top\}\in \mathcal U$ since $(i_1, j_1, l_1)\leq (i_1+1, l_1, \infty)$ and $\mathcal U$ is an upper set. Then by induction we have $\ua (i_2, j^{\prime}, \infty)\cup\{\top\}\in \mathcal U$ for some $j^{\prime}\in \mathbb{N}$. Again, since $\bigvee_{l\in \mathbb{N}}(\ua (i_2, j^{\prime}, \infty)\cup \{\top\})=\bigcap_{l\in \mathbb{N}}(\ua (i_2, j^{\prime}, \infty)\cup \{\top\})=\ua (i_2, j^{\prime}, \infty)\cup \{\top\}\in \mathcal U\in \sigma(\mathsf{K}((\Sigma~\!\!\mathcal L)_{\top}))$ and $\bigvee_{l\in \mathbb{N}}(\ua (i_2, j_2, \infty)\cup \{\top\})=\bigcap_{l\in \mathbb{N}}(\ua (i_2, j_2, \infty)\cup \{\top\})=\ua (i_2, j_2, \infty)\cup \{\top\}\in \mathcal V\in \sigma(\mathsf{K}((\Sigma~\!\!\mathcal L)_{\top}))$, we have some $k_1, k_2\in\mathbb{N}$ such that $\ua (i_2, j^{\prime}, k_1)\in \mathcal U$ and $\ua (i_2, j_2, k_2)\in \mathcal V$. Take $m=\mathrm{max}\{k_1, k_2\}$. Then $\ua (i_2, m, \infty)\cup\{\top\}\in \mathcal A\bigcap \mathcal U\bigcap V$. Thus $\mathcal A\in \ir_c(\Sigma~\!\!\mathsf{K}((\Sigma~\!\!\mathcal L)_{\top}))$.

{\bf Claim 1:} $\mathcal A$ has no largest element.

Clearly, $\{\ua (i, j, \infty)\cup \{\top\} : i, j\in\mathbb{N}\}$ is the set of all maximal elements of $\mathcal A$ and hence $\mathcal A$ has no largest element.

By Claims 1-3, the Scott power space $\Sigma~\!\!\mathsf{K}((\Sigma~\!\!\mathcal L)_{\top})$ is non-sober.
\end{proof}

We know that every $T_2$ space is sober and hence its Scott power space is well-filtered by Theorem \ref{wf imply Scott wf}. In the next section we will show that Soctt power spaces of locally compact (especially, compact) $T_2$ spaces are sober (see Corollary \ref{Scott power space of locally compact T2 space is sober} below).

By Theorem \ref{Schalk-Heckman-Keimel theorem}, Corollary \ref{wf space Vietoris less Scott}, Theorem \ref{Scott power space of a sober space is non-sober} and Corollary \ref{Scott power space of locally compact T2 space is sober} below, we naturally pose the following question.

\begin{question}\label{T2 Scott power space sober?}  For a $T_2$ space $X$, is the Scott power space $\Sigma~\!\!\mk (X)$ sober?
\end{question}

\section{Local compactness, first-countability and sobriety of Scott power spaces}

In this section, we investigate the conditions under which the Scott power space of a sober space is still sober. We will see that Question 3 is related to the investigation of conditions under which the upper Vietoris topology coincides with the Scott topology on $\mk (X)$, and further it is closely related to the local compactness and first-countability of $X$.

First, by Corollary \ref{WFcorcomp-sober}, Corollary \ref{first-countable WF is sober} and Theorem \ref{wf imply Scott wf}, we get the following.

\begin{corollary}\label{wf Scott CI imply sober} If $X$ is a well-filtered space for which the Scott power space $\Sigma~\!\!\mk (X)$ is first-countable or core-compact (especially, locally compact), then $\Sigma~\!\!\mk (X)$ is sober.
\end{corollary}

For the local compactness of Smyth power spaces, we have the following.

\begin{lemma}\label{Smyth power space locall compact} (\cite[Theorem 3.1]{LJ-2020})	For a $T_0$ space, the following conditions are equivalent:
\begin{enumerate}[\rm (1)]
        \item $X$ is locally compact.
        \item $P_S(X)$ is core-compact.
        \item $P_S(X)$ is locally compact.
        \item $P_S(X)$ is locally hypercompact.
        \item $P_S(X)$ is a $c$-space.

\end{enumerate}
\end{lemma}

The following corollary follows directly from Proposition \ref{erne1} and Lemma \ref{Smyth power space locall compact}.

\begin{corollary}\label{LC imply Smyth power space is DC}
	For a locally compact $T_0$ space $X$, the Smyth power space $P_S(X)$ is a $\mathsf{DC}$ space.
\end{corollary}

 Concerning the Scott power space of a locally compact $T_0$ space, we have the following question.

\begin{question}\label{LC imply Scott power space is WD} For a locally compact $T_0$ space $X$, is the Scott power space $\Sigma~\!\!\mk (X)$ a Rudin space or a $\mathsf{WD}$ space?
\end{question}

\begin{proposition}\label{Scott power space of locally compact sober space is sober}  Let $X$ be a locally compact sober space. Then
\begin{enumerate}[\rm (1)]
\item the Scott power space of $X$ and the Symth power space of $X$ coincide, that is, $\Sigma~\!\!\mathsf{K}(X)=P_S(X)$.
\item $\mk (X)$ is a continuous domain.
\item $\Sigma~\!\!\mk (X)$ is a sober $c$-space.
\end{enumerate}
\end{proposition}

\begin{proof} By Corollary \ref{wf space Vietoris less Scott} and Lemma \ref{local comp imply V topol finer than Scott topol}, $\Sigma~\!\!\mathsf{K}(X)=P_S(X)$. By \cite[Proposition I-1.24.2]{redbook}, $\mathsf{K}(X)$ is a continuous semilattice, and hence by Theorem \ref{algebraic is continuous} and Proposition \ref{quasicontinuous domain is sober}, $\Sigma~\!\!\mk (X)$ is a sober $c$-space.
\end{proof}

\begin{corollary}\label{Scott power space of locally compact T2 space is sober}  If $X$ is a locally compact $T_2$ (especially, a compact $T_2$) space, then
\begin{enumerate}[\rm (1)]
\item $\Sigma~\!\!\mathsf{K}(X)=P_S(X)$.
\item $\mk (X)$ is a continuous domain.
\item $\Sigma~\!\!\mk (X)$ is a sober $c$-space.
\end{enumerate}
\end{corollary}

By Theorem \ref{algebraic is continuous}, Proposition \ref{quasicontinuous domain is sober} and Proposition \ref{Scott power space of locally compact sober space is sober}, we have the following corollary.

\begin{corollary}\label{quasicontinuous domain Scott power space sober c-space} Let $P$ be a quasicontinuous domain. Then
 \begin{enumerate}[\rm (1)]
 \item the upper Vietoris topology agrees with the Scott topology on $\mk (\Sigma~\!\!P)$.
 \item $\mk (\Sigma~\!\!P)$ is a continuous semilattice.
 \item the Scott power space $\Sigma~\!\!\mk(\Sigma~\!\!P)$ is a sober $c$-space.
 \end{enumerate}
\end{corollary}

Now we discuss the first-countability of the Scott power spaces. First, for the Smyth power spaces and sobrifications of $T_0$ spaces, we have the following conclusion.

\begin{proposition}\label{Smyth CII} (\cite{Brecht, XSXZ-2021, XY20})  For a $T_0$ space, the following conditions are equivalent:
\begin{enumerate}[\rm (1)]
\item $X$ is second-countable.
\item $P_S(X)$ is second-countable.
\item $X^s$ is second-countable.
\end{enumerate}
\end{proposition}

Since first-countability is a hereditary property, by Remark \ref{eta continuous} and Remark \ref{xi embdding}, we get the following result.

\begin{proposition}\label{soberification first-countable implies X is also}  Let $X$ be a $T_0$ space. If $X^s$ is first-countable or $P_S(X)$ is first-countable, then $X$ is first-countable.
\end{proposition}

Example \ref{Scott sober not implies X is sober} shows that unlike the Smyth power space, the first-countability of the Scott power space of a $T_0$ space $X$ does not imply the first-countability of $X$ in general.

The converse of Proposition \ref{soberification first-countable implies X is also} does not hold in general, as shown in Example \ref{first-countable omega WF is not sober} and the following example. It also shows that even for a compact Hausdorff first-countable space $X$, the Scott power space of $X$ and the Smyth power space of $X$ may not be first-countable.

 There is even a $T_0$ space $X$ for which the Scott power space $\Sigma~\!\!\mathsf{K}(X)$ is second-countable but $X$ is not fist-countable (see Example \ref{Scott power space CII but X not CI} below). So for Scott power spaces, the analogous results to Proposition \ref{Smyth CII} and Proposition \ref{soberification first-countable implies X is also} do not hold.

\begin{example}\label{X CI but Smyth not}  Consider in the plane $\mathbb{R}^2$ two concentric circles $C_i=\{(x,y)\in \mathbb{R}^2 :
x^2+y^2=i\}$, where $i=1, 2$, and their union $X=C_1\cup C_2$; the projection of $C_1$ onto $C_2$ from
the point $(0,0)$ is denoted by $p$. On the set $X$ we generate a topology by defining
a neighbourhood system $\{B(z): z\in X\}$ as follows: $B(z)=\{{z}\}$ for $z\in C_2$ and
$B(z)=\{U_j(z): j\in \mathbb{N}\}$ for $z\in C_1$, where $U_j=V_j\bigcup p(V_j\setminus \{z\})$ and $V_j$ is the arc of $C_1$ with center at $z$ and of length $1/j$. The space $X$ is called the \emph{Alexandroff double circle} (cf. \cite{Engelking}). The following conclusions about $X$ are known (see, for example, \cite[Example 3.1.26]{Engelking}).
\begin{enumerate}[\rm (a)]
\item  $X$ is Hausdorff and first-countable.
 \item  $X$ is compact and locally compact.
 \item  $X$ is not separable, and hence not second-countable.
 \item $C_1$ is a compact subspace of $X$.
 \item $C_2$ is a discrete subspace of $X$.
\end{enumerate}

\noindent There is no countable base at $C_1$ in $P_S(X)$. Thus $P_S(X)$ is not first-countable. For details, see \cite[Example 4.4]{XY20}. By Corollary \ref{Scott power space of locally compact T2 space is sober}, $\Sigma~\!\!\mathsf{K}(X)=P_S(X)$, whence the Scott power space of $X$ is not first-countable.
\end{example}

\begin{proposition}\label{min Compact countable is Smth CI} (\cite[Proposition 4.5]{XY20}) Let $X$ be a first-countable $T_0$ space. If $\mathrm{min}(K)$ is countable for any $K\in \mk (X)$, then
$P_S(X)$ is first-countable.
\end{proposition}

\begin{proposition}\label{metric space is Smth CI} For a metric space $(X, d)$, $P_S((X, d))$ is first-countable.
\end{proposition}
\begin{proof} For $K\in \mk ((X, d))$, let $\mathcal B_K=\{B(K, 1/n) : n\in \mathbb{N}\}$. Then by Proposition \ref{metric space compact sets}, $\mathcal B_K=\{B(K, 1/n) : n\in \mathbb{N}\}$ is a countable base at $K$ in $P_S((X, d))$. Thus $P_S((X, d))$ is first-countable.
\end{proof}

For a countable $T_0$ space $X$, it is easy to see that $X$ is second-countable iff $X$ is first-countable. Indeed, let $X=\{x_n : n\in \mn\}$. If $X$ is first-countable, then for each $n\in \mn$, there is a countable base $\mathcal B_n$ at $x_n$. Let $\mathcal B=\bigcup_{n\in \mn}\mathcal B_n$. Then $\mathcal B$ is a countable base of $X$. Thus $X$ is second-countable. Therefore, by Proposition \ref{Smyth CII} and Proposition \ref{soberification first-countable implies X is also}, we have the following.

\begin{corollary}\label{countable first-countable is second-countable} For a countable $T_0$ space $X$, the following conditions are equivalent:
\begin{enumerate}[\rm (1)]
\item $X$ is first-countable.
\item $X$ is second-countable.
\item $X^s$ is first-countable.
\item $X^s$ is second-countable.
\item $P_S(X)$ is first-countable.
\item $P_S(X)$ is second-countable.

\end{enumerate}
\end{corollary}

It is worth noting that the Scott topology on a countable complete lattice may not be first-countable, see \cite[Example 4.8]{xu-shen-xi-zhao2}.

By Proposition \ref{sober equiv using omega RD and omega WD}, Theorem \ref{sobrifcaltion first-countable is Rudin}, Proposition \ref{Smyth CII} and Corollary \ref{countable first-countable is second-countable}, we deduce the following two results.

\begin{corollary}\label{second-countable is omega Rudin}(\cite[Corollary 5.7 and Corollary 5.8]{XSXZ-2021}) Every second-countable (especially, countable first-countable) $T_0$ space is an $\omega$-Rudin space.
\end{corollary}

\begin{corollary}\label{secound-countable omega WF is sober} Every second-countable (especially, countable first-countable) $\omega$-well-filtered space is sober.
\end{corollary}

For a $T_0$ space $X$ with a first-countable Smyth power space, we have a similar result to Lemma \ref{local comp imply V topol finer than Scott topol}.

\begin{lemma}\label{Smyth first-countable imply V topol finer than Scott topol} Let $X$ be a $T_0$ space for which the Smyth power space $P_S(X)$ is first-countable. Then the Scott topology is coarser than the upper Vietoris topology on $\mathsf{K}(X)$.
\end{lemma}
\begin{proof} See the proof of \cite[Theorem 5.7]{XY20}.
\end{proof}

The following conclusion is straightforward from Theorem \ref{Smythwf}, Corollary \ref{wf space Vietoris less Scott}, Corollary \ref{first-countable WF is sober} and Lemma \ref{Smyth first-countable imply V topol finer than Scott topol}.

\begin{corollary}\label{wf Smyth CI S=V} (\cite[Theorem 5.7]{XY20}) Let $X$ be a well-filtered space for which the Smyth power space $P_S(X)$ is first-countable. Then
\begin{enumerate}[\rm (1)]
\item the upper Vietoris topology agrees with the Scott topology on $\mk (X)$.
\item the Scott power space $\Sigma~\!\!\mk (X)$ is a first-countable sober space.
\end{enumerate}
\end{corollary}

By Proposition \ref{min Compact countable is Smth CI} and Corollary \ref{wf Smyth CI S=V}, we obtain the following.

\begin{corollary}\label{countable countable CI S=V} (\cite[Corollary 5.10]{XY20}) Let $X$ be a first-countable well-filtered space $X$ in which
all compact subsets are countable (especially, $|X|\leq\omega$). Then
\begin{enumerate}[\rm (1)]
\item the upper Vietoris topology agrees with the Scott topology on $\mk (X)$.
\item the Scott power space $\Sigma~\!\!\mk (X)$ is a first-countable sober space.
\end{enumerate}
\end{corollary}

Let $X_{coc}$ be the space in Example \ref{Scott sober not implies X is sober}. Then $X_{coc}$ is well-filtered and not first-countable, and the Scott power space $\Sigma~\!\!\mk (X)$ is a first-countable sober $c$-space, but $\sigma (\mk (X_{coc})\nsubseteq \mathcal O(P_S(X_{coc}))$.

By Example \ref{Scott sober not implies X is sober}, Lemma \ref{Smyth first-countable imply V topol finer than Scott topol} and Corollary \ref{wf Smyth CI S=V}, we naturally pose the following four questions.

\begin{question}\label{first-countable imply V topol finer than Scott topol} For a first-countable $T_0$ space $X$, is the Scott topology coarser than the upper Vietoris topology on $\mathsf{K}(X)$?
\end{question}

\begin{question}\label{wf CI S=V} For a first-countable well-filtered (or equivalently, a first-countable sober) space $X$, does the upper Vietoris topology and the Scott topology on $\mk (X)$ coincide?
\end{question}

\begin{question}\label{T2 CI S=V} For a first-countable $T_2$ space $X$, does the upper Vietoris topology and the Scott topology on $\mk (X)$ coincide?
\end{question}

\begin{question}\label{wf CI Scott power space is sober} Is the Scott power space of a first-countable well-filtered (or equivalently, a first-countable sober) space sober?
\end{question}

Since every metric space is $T_2$ (and hence sober), by Proposition \ref{Scott power space of locally compact sober space is sober}, Proposition \ref{metric space is Smth CI} and Corollary \ref{wf Smyth CI S=V}, we get the following conclusion.

\begin{corollary}\label{metric space Scott power space first-countable} Let $(X, d)$ be a metric space. Then
 \begin{enumerate}[\rm (1)]
 \item the upper Vietoris topology agrees with the Scott topology on $\mk ((X, d))$.
 \item the Scott power space $\Sigma~\!\!\mk((X, d))$ is a first-countable sober space.
 \end{enumerate}

\noindent If, in addition, $(X, d)$ is locally compact (especially, compact), then
\vskip 0.2cm
\noindent \emph{(3)} $\mk ((X, d))$ is a continuous semilattice.

\noindent \emph{(4)}  the Scott power space $\Sigma~\!\!\mk((X, d))$ is a $c$-space.

\end{corollary}

The following two conclusions follow directly from Proposition \ref{Smyth CII}, Lemma \ref{Smyth first-countable imply V topol finer than Scott topol} and Corollary \ref{wf Smyth CI S=V}.

\begin{corollary}\label{second-countable imply V topol finer than Scott topol} Let $X$ be a second-countable $T_0$ space. Then the Scott topology is coarser than the upper Vietoris topology on $\mathsf{K}(X)$.
\end{corollary}

\begin{corollary}\label{second-countable wf imply V topol=Scott topol} Let $X$ be a second-countable well-filtered space (or equivalently, a second-countable sober space). Then
\begin{enumerate}[\rm (1)]
\item the Scott topology agrees with the upper Vietoris topology on $\mathsf{K}(X)$.
 \item the Scott power space of $X$ is a second-countable sober space.
 \end{enumerate}
\end{corollary}

The following example shows that there is a countable Hausdorff space $X$ for which the Scott power space $\Sigma ~\!\!\mk (X)$ is second-countable but $X$ is not first-countable (and hence $P_S(X)$ is not first-countable).

\begin{example}\label{Scott power space CII but X not CI} Let $p$ be a point in $\beta (\mathbb{N})\setminus \mathbb{N}$, where $\beta (\mathbb{N})$ is the Stone-C\v ech compactification of the discrete space of natural numbers, and consider on $X=\mathbb{N}\cup\{p\}$ the
induced topology (cf. \cite[Example II-1.25]{redbook}). Then
\begin{enumerate}[\rm (a)]
\item $|X|=\omega$ and $X$ is a non-discrete Hausdorff space and hence a sober space.
\item $\mk (X)=X^{(<\omega)}\setminus \{\emptyset\}$ and $\mathrm{int}K=\emptyset$ for each $K\in \mk (X)$. So $X$ is not locally compact.
\item $\mk (X)$ is a Noetherian poset and $|\mk (X)|=\omega$. Hence the Scott power space $\Sigma~\!\!\mk (X)$ is a second-countable sober $c$-space.

Clearly, $\mk (X)=X^{(<\omega)}\setminus \{\emptyset\}$ (with the Smyth order) is Noetherian (and hence algebraic) and $|\mk (X)|=\omega$ since $|X|=\omega$. Therefore, $\sigma(\mk (X))=\alpha(\mk (X))$ and $\{\uparrow_{\mk (X)} F : F\in X^{(<\omega)}\setminus \{\emptyset\}\}$ is a countable base of $\Sigma \mk (X)$. By Theorem \ref{algebraic is continuous} and Proposition\ref{quasicontinuous domain is sober}, $\Sigma \mk (X)$ is a sober $c$-space.

 \item the upper Vietoris topology and the Scott topology on $\mk (X)$ does not coincide, or more precisely, $\sigma (\mk (X))\nsubseteq \mathcal O(P_S(X))$.

     By Corollary \ref{wf space Vietoris less Scott}, $\mathcal O(P_S(X))\subseteq \sigma (\mk (X))$. Clearly, for any $F\in X^{(<\omega)}\setminus \{\emptyset\}$, $\uparrow_{\mk (X)} F\in \sigma(\mk (X))$ but $\uparrow_{\mk (X)} F\not\in \mathcal O(P_S(X))$.

 \item  Neither $X$ nor $P_S(X)$ is first-countable.

     By (d), Proposition \ref{min Compact countable is Smth CI} and Lemma \ref{Smyth first-countable imply V topol finer than Scott topol}, neither $P_S(X)$ nor $X$ is first-countable (cf. \cite[Corollary 3.6.17]{Engelking}).
     \end{enumerate}
\end{example}

The above example also shows that if the Smyth power space is replaced with the Scott power space in the conditions of Lemma \ref{Smyth first-countable imply V topol finer than Scott topol} and Corollary \ref{wf Smyth CI S=V}, the analogous results to Lemma \ref{Smyth first-countable imply V topol finer than Scott topol} and Corollary \ref{wf Smyth CI S=V} do not hold.

By Proposition \ref{Smyth CII}, Lemma \ref{Smyth first-countable imply V topol finer than Scott topol}, Corollary \ref{wf Smyth CI S=V} and Corollary \ref{second-countable wf imply V topol=Scott topol}, we raise the following question.

\begin{question}\label{X CII implies Scott power space CII} For a second-countable $T_0$ space $X$, is the Scott power space of $X$ second-countable?
\end{question}

\section{Rudin property and well-filtered determinedness of Smyth power sapces and Scott power spaces}

Firstly, we discuss the Rudin property and well-filtered determinedness of Smyth power spaces. The following result was proved in \cite{XSXZ-2020}.

\begin{proposition}\label{SmythWD} (\cite[Theorem 7.21]{XSXZ-2020}) Let $X$ be a $T_0$ space. If $P_S(X)$ is well-filtered determined, then $X$ is well-filtered determined.
\end{proposition}

By Theorem \ref{Smythwf} and Theorem \ref{soberequiv}, we have the following.

\begin{proposition}\label{wf Smyth RD or WD = X Rudin or WD} Let $X$ be a well-filtered space. Then the following conditions are equivalent:
\begin{enumerate}[\rm (1)]
\item $X$ is a Rudin space.
\item $X$ is a $\mathsf{WD}$ space.
\item $P_S(X)$ is a Rudin space.
\item $P_S(X)$ is a $\mathsf{WD}$ space.
\end{enumerate}
\end{proposition}

It is still not known whether the converse of Proposition \ref{SmythWD} holds (that is, whether the Smyth power space $P_S(X)$ of a well-filtered determined $T_0$ space $X$ is well-filtered determined) (see \cite[Question 8.6]{XSXZ-2020}).

\begin{theorem}\label{Smyth RD}  Let $X$ be a $T_0$ space. If $P_S(X)$ is a Rudin space, then $X$ is a Rudin space.
\end{theorem}

\begin{proof} Let $A\in\ir_c(X)$. Then by Lemma \ref{closure in Smyth power space}, $\overline{\xi_X(A)}=\Diamond A\in\ir_c(P_S(X))$, where $\xi_X : X \longrightarrow P_S(X)$ is the canonical embedding (see Remark \ref{xi embdding}). Since $P_S(X)$ is a Rudin space, there is a filtered family $\{\mathcal{K}_d : d\in D\}\subseteq \mathbf{K}(P_S(X))$ such that $\Diamond A\in m(\{\mathcal{K}_d : d\in D\})$. For each $d\in D$, let $K_d=\bigcup \mathcal{K}_d$. Then by Lemma \ref{K union}, $\{K_d : d\in D\}\subseteq \mathbf{K}(X)$ is filtered. Clearly, $A\in M(\{K_d : d\in D\})$. For any proper closed subset $B$ of $A$, we have that $\Diamond B\in\mathcal C(P_S(X))$ and $\Diamond B$ is a proper closed subset of $\Diamond A$ (for any $a\in A\setminus B$, $\ua a\in \Diamond A\setminus \Diamond B$). By the minimality of $\Diamond A$, there is a $d\in D$ such that $\Diamond B\cap \mathcal{K}_d=\emptyset$, and consequently, $B\cap K_d=\emptyset$. Thus $B\notin M(\{K_d : d\in D\})$, and hence $A\in m(\{K_d : d\in D\})$.
\end{proof}

\begin{question}\label{Smyth RD question} Is the Smyth power space $P_S(X)$ of a Rudin space $X$ still a Rudin space?
\end{question}

Now we discuss the Rudin property and well-filtered determinedness of Scott power spaces.

First, even for a sober space $X$ (whence it is both a Rudin space and a $\mathsf{WD}$ space by Theorem \ref{soberequiv}), its Scott power space may not be a $\mathsf{WD}$ space (and hence not a $\mathsf{WD}$ space). Indeed, let $(\Sigma~\!\!\mathcal L)_{\top}$ be as in Theorem \ref{Scott power space of a sober space is non-sober}. Then $(\Sigma~\!\!\mathcal L)_{\top}$ is a sober space. By Theorem \ref{wf imply Scott wf} and Theorem \ref{Scott power space of a sober space is non-sober}, the Scott power space $\Sigma~\!\!\mathsf{K}((\Sigma~\!\!\mathcal L)_{\top})$ is well-filtered but non-sober. Hence by Theorem \ref{soberequiv}, $\Sigma~\!\!\mathsf{K}((\Sigma~\!\!\mathcal L)_{\top})$ is neither a Rudin space nor a $\mathsf{WD}$ space.

Conversely, Example \ref{Scott sober not implies X is sober} shows that there is a well-filtered space $X$ such that
 \begin{enumerate}[\rm (a)]
 \item the Scott power space $\Sigma~\!\!\mk (X)$ is a first-countable sober $c$-space and hence $\Sigma~\!\!\mk (X)$ is both Rudin and $\mathsf{WD}$.
 \item $X$ is neither a Rudin space nor a $\mathsf{WD}$ space.
 \item the Smyth power space $P_S(X)$ is neither a Rudin space nor a $\mathsf{WD}$ space.
 \end{enumerate}

 Then we investigate some sufficient conditions under which the well-filtered determinedness (resp. the Rudin property) of Scott power space of a $T_0$ space $X$ implies that of $X$.

\begin{definition}\label{property S} A $T_0$ space $X$ is said to have \emph{property S} if for each $A\in \ir_c(X)$, $\{\ua a : a\in A\}\in \ir (\Sigma~\!\! \mk (X))$ or $\Diamond A\in \ir_c(\Sigma~\!\! \mk (X))$. A poset $P$ is said to have property S if $\Sigma~\!\! P$ has property S.
\end{definition}

\begin{remark}\label{xi continuous implies property S} Let $X$ be a $T_0$ space and $A\in \ir_c(X)$.
\begin{enumerate}[\rm (1)]
\item Since $\xi_X : X \longrightarrow P_S(\mk (X)), x\mapsto \ua x$, is continuous, $\{\ua a : a\in A\}\in \ir (P_S(X))$ and $\cl_{\mathcal O(P_S(X))}\{\ua a : a\in A\}=\Diamond A\in \ir_c(P_S(X))$ by Lemma \ref{irrimage} and Lemma \ref{closure in Smyth power space}.

\item If $\xi_X^\sigma : X \longrightarrow \Sigma~\!\! \mk (X), x\mapsto \ua x$, is continuous, then by Lemma \ref{irrimage}, $X$ has property S.

\item If $\sigma (\mk (X))\subseteq \mathcal O(P_S(X))$, then $\xi_X^\sigma : X \longrightarrow \Sigma~\!\! \mk (X)$ is continuous by Remark \ref{xi embdding}, and hence $X$ has property S.

\item For a poset $P$, by Lemma \ref{Scott continuous equiv} and Lemma \ref{Kmeet}, the mapping $\xi_P^\sigma : \Sigma~\!\! P  \longrightarrow \Sigma~\!\! \mk (\Sigma~\!\! P), x\mapsto \ua x$, is continuous. Therefore, $P$ has property S.
    \end{enumerate}
 \end{remark}

\begin{proposition}\label{Scott RD is RD}  Suppose that a $T_0$ space $X$ has property S and $\mathcal O(P_S(X))\subseteq\sigma (\mk (X))$. If  $\Sigma~\!\!\mathsf{K}(X)$ is a Rudin space, then $X$ is a Rudin space.
\end{proposition}

\begin{proof} Let $A\in\ir_c(X)$. Then by the property S of $X$, $\{\ua a : a\in A\}\in \ir (\Sigma~\!\! \mk (X))$ or $\Diamond A\in \ir_c(\Sigma~\!\! \mk (X))$.

\textbf{Case 1:} $\{\ua a : a\in A\}\in \ir (\Sigma~\!\! \mk (X))$.

Since $\Sigma~\!\!\mathsf{K}(X)$ is a Rudin space, there is a filtered family $\{\mathcal{K}_d : d\in D\}\subseteq \mathsf{K}(\Sigma ~\!\!\mathsf{K}(X))$ such that $\cl_{\sigma (\mk (X))}\{\ua a : a\in A\}\in m(\{\mathcal{K}_d : d\in D\})$. As $\mathcal O(P_S(X))\subseteq \sigma (\mathsf{K}(X))$, we have that $\cl_{\sigma (\mk (X))}\{\ua a : a\in A\}\subseteq \Diamond A\in \mathcal C(P_S(X))\subseteq \mathcal C(\Sigma~\!\!\mk (X))$ and $\{\mathcal{K}_d : d\in D\}\subseteq \mathsf{K}(P_S(X))$. Therefore, $\Diamond A\in M(\{\mathcal{K}_d : d\in D\})$. For each $d\in D$, let $K_d=\bigcup \mathcal{K}_d$. Then by Lemma \ref{K union}, $\{K_d : d\in D\}\subseteq \mathsf{K}(X)$ is filtered. Since $\Diamond A\in M(\{\mathcal{K}_d : d\in D\})$, $A\in M(\{K_d : d\in D\})$. Now we show that $A\in m(\{K_d : d\in D\})$. Suppose that $B$ is a proper closed subset $B$ of $A$. Then there is $a\in A\cap (X\setminus B)$, and hence $\ua a\in \Box (X\setminus B)\in \mathcal O(P_S(X)\subseteq \sigma (\mk (X))$. Clearly, $\{\ua b : b\in B\}\cap \Box (X\setminus B)=\emptyset$, and consequently, $\ua a\notin \cl_{\sigma (\mk (X))}\{\ua b : b\in B\}$. Therefore, $\cl_{\sigma (\mk (X))}\{\ua b : b\in B\}$ is a proper subset of $\cl_{\sigma (\mk (X))}\{\ua a : a\in A\}$. By $\cl_{\sigma (\mk (X))}\{\ua a : a\in A\}\in m(\{\mathcal{K}_d : d\in D\})$, there is $d_0\in D$ such that $\cl_{\sigma (\mk (X))}\{\ua b : b\in B\}\cap \mathcal K_{d_0}=\emptyset$, and hence $\{\ua b : b\in B\}\cap \mathcal K_{d_0}=\emptyset$. Since $\mathcal K_{d_0}=\ua_{\mk (X)} \mathcal K_{d_0}$, we have that $B\cap K_{d_0}=B\cap (\bigcup \mathcal K_{d_0})=\emptyset$. Thus $A\in m(\{K_d : d\in D\})$.

\textbf{Case 2:} $\Diamond A\in \ir_c(\Sigma~\!\! \mk (X))$.

Since $\Sigma~\!\!\mathsf{K}(X)$ is a Rudin space, there is a filtered family $\{\mathcal{K}_d : d\in D\}\subseteq \mathbf{K}(\Sigma~\!\!\mathsf{K}(X))$ such that $\Diamond A\in m(\{\mathcal{K}_d : d\in D\})$. As carried out in the proof of Case 1, $A$ is a Rudin set of $X$.

Thus $X$ is a Rudin space.
\end{proof}

\begin{corollary}\label{WF Property S Scott RD is sober}  Suppose that $X$ is a well-filtered space with property S. If $\Sigma~\!\!\mathsf{K}(X)$ is a $\mathsf{WD}$ space (especially, a Rudin space), then both $\Sigma~\!\!\mk (X)$ and $X$ are sober.
\end{corollary}
\begin{proof} By Theorem \ref{wf imply Scott wf}, $\Sigma~\!\!\mathsf{K}(X)$ is well-filtered. As $\Sigma~\!\!\mathsf{K}(X)$ is $\mathsf{WD}$ (if $\Sigma~\!\!\mathsf{K}(X)$ is Rudin, then by Proposition \ref{SDRWspacerelation} it is  $\mathsf{WD}$), by Theorem \ref{soberequiv}, $\Sigma~\!\!\mk (X)$ is sober. Hence, by Theorem \ref{soberequiv} and Corollary \ref{wf space Vietoris less Scott}, $\Sigma~\!\!\mk (X)$ is Rudin and $\mathcal O(P_S(X))\subseteq\sigma (\mk (X))$, and consequently, $X$ is Rudin by Proposition \ref{Scott RD is RD}. It follows from Theorem \ref{soberequiv} that $X$ is sober.

\end{proof}

By Remark \ref{xi continuous implies property S} and Corollary \ref{WF Property S Scott RD is sober}, we have the following corollary.

\begin{corollary}\label{WF Scott RD continuous is sober} Let $X$ be a well-filtered space. If $\xi_X^\sigma : X \longrightarrow \Sigma~\!\! \mk (X)$ is continuous and $\Sigma~\!\!\mk (X)$ is a $\mathsf{WD}$ space (especially, a Rudin space), then both $\Sigma~\!\!\mk (X)$ and $X$ are sober.
\end{corollary}

As an immediate corollary of Corollary \ref{WF Scott RD continuous is sober} we get one the main results of \cite{xuxizhao}.

\begin{corollary}\label{WF Scott sober is sober} (\cite[Theorem 2]{xuxizhao}) Suppose that $X$ is a well-filtered space and $\xi_X^\sigma : X \longrightarrow \Sigma~\!\! \mk (X)$ is continuous. If $\Sigma~\!\!\mathsf{K}(X)$ is sober, then $X$ is sober. Therefore, if $X$ is non-sober, then its Scott power space $\Sigma~\!\!\mathsf{K}(X)$ is non-sober.
\end{corollary}

Example \ref{Scott sober not implies X is sober} shows that when $X$ lacks the property S or the continuity of $\xi_X^\sigma : X \longrightarrow \Sigma~\!\! \mk (X)$, Proposition \ref{Scott RD is RD}, Corollary \ref {WF Property S Scott RD is sober}, Corollary \ref{WF Scott RD continuous is sober} and Corollary \ref{WF Scott sober is sober} may not hold.

By Remark \ref{xi continuous implies property S}, Proposition \ref{Scott RD is RD}, Corollary \ref{WF Property S Scott RD is sober} and Corollary \ref{WF Scott sober is sober}, we deduce the following three corollaries.

\begin{corollary}\label{Scott P  RD is P RD}  Let $P$ be a poset. If $\mathcal O(P_S(\Sigma~\!\!P))\subseteq\sigma (\mk (\Sigma~\!\!P))$ and   $\mathsf{K}(\Sigma~\!\!P)$ is a Rudin poset, then $P$ is a Rudin dcpo.
\end{corollary}

\begin{corollary}\label{WF and Scott P  RD  P is sober}  Let $P$ be a well-filtered dcpo. If $\mathsf{K}(\Sigma~\!\!P)$ is a $\mathsf{WD}$ dcpo (especially, a Rudin dcpo), then both $\mk (\Sigma~\!\!P)$ and $P$ are sober dcpos.
\end{corollary}

\begin{corollary}\label{WF and Scott P  sober P is sober}  Let $P$ be a well-filtered dcpo. If $\mathsf{K}(\Sigma~\!\!P)$ is a sober dcpo, then $P$ is a sober dcpo. Therefore, if $P$ is not a sober dcpo, then  $\mathsf{K}(\Sigma~\!\!P)$ is not a sober dcpo.
\end{corollary}

\begin{example}\label{Isbell lattice not Rudin}  Let $L$ be the Isbell's lattice constructed in \cite{isbell}. Then
\begin{enumerate}[\rm (a)]
\item $L$ is not a sober dcpo (see \cite{isbell}).
\item $L$ is a well-filtered dcpo by Proposition \ref{Xi-Lawson result 1}
\item $L$ is neither a Rudin dcpo nor a $\wdd$ dcpo by (a)(b) and Corollary \ref{sober dcpo equivalent}.
\item $\mathsf{K}(\Sigma~\!\!L)$ is a well-filtered dcpo by Theorem \ref{wf imply Scott wf}.
\item $\mathsf{K}(\Sigma~\!\!L)$ is a spatial frame (see \cite[Lemma 1]{xuxizhao}).
\item $\mathsf{K}(\Sigma~\!\!L)$ is not a sober dcpo by (a) and Corollary \ref{WF and Scott P  sober P is sober}.
\item $\mathsf{K}(\Sigma~\!\!L)$ is neither a Rudin dcpo nor a $\wdd$ dcpo by (a)(b) and Corollary \ref{WF and Scott P  RD P is sober}

\end{enumerate}
\end{example}

\begin{proposition}\label{Scott WD implies X is WD} Suppose that $X$ is a $T_0$ space for which $\sigma (\mk (X))\subseteq\mathcal O(P_S(X))$. If $\Sigma~\!\! \mk (X)$ is well-filtered determined, then $X$ is well-filtered determined.
\end{proposition}
\begin{proof} By Remark \ref{xi continuous implies property S}, $X$ has property S. Let $A\in\ir_c(X)$, $Y$ a well-filtered space and $f:X\longrightarrow Y$  a continuous mapping. Then by $\sigma (\mk (X))\subseteq\mathcal O(P_S(X))$, Lemma \ref{PS functor} and Theorem \ref{wf imply Scott wf}, $\Sigma~\!\! \mk (Y)$ is well-filtered and $P_S^{\sigma}(f): P_S(X)\longrightarrow \Sigma~\!\! \mk (Y)$ is continuous, where $P_S^\sigma (f)(K)=\ua f(K)$ for all $K\in \mk (X)$. By assumption, $\{\ua a : a\in A\}\in \ir (\Sigma~\!\! \mk (X))$ or $\Diamond A\in \ir_c(\Sigma~\!\! \mk (X))$, and hence by the well-filtered determinedness of $\Sigma~\!\! \mk (Y)$ and the continuity of $P_S^\sigma (f)$, there exists a unique $Q\in \mk(Y)$ such that $\overline{\{\ua f(a) : a\in A\}}=\overline{P_S^\sigma(f)(\{\ua a : a\in A\})}=\overline{\{Q\}}=\downarrow_{\mk (Y)} Q$ in $\Sigma~\!\! \mk (Y)$ or $\overline{P_S^\sigma(f)(\Diamond A)}=\overline{\{Q\}}=\downarrow_{\mk (Y)} Q$ in $\Sigma~\!\! \mk (Y)$. For the case of $\overline{P_S^\sigma(f)(\Diamond A)}=\overline{\{Q\}}$ in $\Sigma~\!\! \mk (Y)$, since $\overline{P_S^\sigma(f)(\Diamond A)}=\overline{P_S^\sigma(f)(\cl_{\mathcal O(P_S(X))}\xi_X (A))}=\overline{P_S^\sigma(f)(\xi_X (A))}=\overline{\{\ua f(a) : a\in A\}}$ in $\Sigma~\!\! \mk (Y)$, we also have $\overline{\{\ua f(a) : a\in A\}}=\overline{\{Q\}}$ in $\Sigma~\!\! \mk (Y)$. Since $Y$ is well-filtered, by Corollary \ref{wf space Vietoris less Scott}, $\mathcal O(P_S(Y))\subseteq \sigma(\mk (Y))$. Hence by Remark \ref{closure A = closre B}, $\cl_{\mathcal O(P_S(X))}\{\ua f(a) : a\in A\}=\cl_{\mathcal O(P_S(X))}\{Q\}$.

 {\bf Claim 1:} $Q$ is supercompact.

 Let $\{V_j:j\in J\}\subseteq\mathcal O(Y)$ with $Q\subseteq \bigcup_{j\in J}V_j$, i.e., $Q\in\Box \bigcup_{j\in J}V_j$. Since $\cl_{\mathcal O(P_S(X))}\{\ua f(a) : a\in A\}=cl_{\mathcal O(P_S(X))}\{Q\}$, we have $\{\ua f(a):a\in A\}\cap \Box \bigcup_{j\in J}V_j\neq\emptyset$. Then there exists $a_0\in A$ and $j_0\in J$ such that $\ua f(a_0)\subseteq V_{j_0}$, and consequently, $\{\ua a : a\in A\}\cap \Box U_{j_0}\neq\emptyset$. By $\cl_{\mathcal O(P_S(X))}\{\ua f(a) : a\in A\}=\cl_{\mathcal O(P_S(X))}\{Q\}$ again, we have $Q\in \Box U_{j_0}$, that is, $Q\subseteq  U_{j_0}$.

 Hence, by \cite[Fact 2.2]{Klause-Heckmann}, there exists $y_Q\in Y$ such that $Q=\ua y_Q$.

 {\bf Claim 2:} $\overline{f(A)}=\overline{\{y_Q\}}$ in $Y$.

 For each $y\in f(A)$, by $\cl_{\mathcal O(P_S(X))}\{\ua f(a) : a\in A\}=cl_{\mathcal O(P_S(X))}\{\ua y_Q\}$, we have that $\ua y\in \cl_{\mathcal O(P_S(X))}\{\ua y_Q\}=$ $\downarrow_{\mk (Y)}\ua y_Q$, whence $\ua y_Q\subseteq \ua y$, i.e., $y\in\overline{\{y_Q\}}$. This implies that $f(A)\subseteq \overline{\{y_Q\}}$. In addition, since $\ua y_Q\in$ $\cl_{\mathcal O(P_S(X))}\{\ua f(a) : a\in A\}\subseteq \Diamond \overline{f(A)}$ , $\ua y_Q\cap\overline{f(A)}\neq\emptyset$. It follows that $y_Q\in \overline{f(A)}$. Therefore, $\overline{f(A)}=\overline{\{y_Q\}}$.

 By Claim 2, $A\in \wdd (X)$, proving that $X$ is well-filtered determined.
 \end{proof}

\begin{corollary}\label{Scott P and Scott WD imply X WD}  For a poset $P$, if $\sigma (\mk (\Sigma~\!\!P))\subseteq \mathcal O(P_S(\Sigma~\!\!P))$ and $\mk (\Sigma~\!\!P))$ is a $\wdd$ poset, then $P$ is a $\wdd$ poset.
\end{corollary}

From Corollary \ref{sober dcpo equivalent} and Corollary \ref{Scott P and Scott WD imply X WD} we deduce the following result.

\begin{corollary}\label{WF dcpo Scott RD or WD is sober dcpo}  If $P$ is a well-filtered dcpo, $\sigma (\mk (\Sigma~\!\!P))\subseteq \mathcal O(P_S(\Sigma~\!\!P))$ and $\mk (\Sigma~\!\!P)$ is a $\wdd$ dcpo (especially, a Rudin dcpo), then both $\mk (\Sigma~\!\!P)$ and $P$ are sober dcpos.
\end{corollary}

Finally, by Lemma \ref{local comp imply V topol finer than Scott topol}, Lemma \ref{Smyth first-countable imply V topol finer than Scott topol} and Proposition \ref{Scott WD implies X is WD}, we get the following two corollaries.

\begin{corollary}\label{X locally comact and Scott WD implies X is WD} If $X$ is a locally compact $T_0$ space and $\Sigma~\!\! \mk (X)$ is well-filtered determined, then $X$ is well-filtered determined.
\end{corollary}

\begin{corollary}\label{Smyth first-countable and Scott WD implies X is WD} Suppose that $X$ is a $T_0$ space for which the Smyth power space $P_S(X)$ is first-countable. If $\Sigma~\!\! \mk (X)$ is well-filtered determined, then $X$ is well-filtered determined.
\end{corollary}

\end{document}